\documentclass{amsart}
\usepackage{amsfonts}
\usepackage{amsthm}
\usepackage{amsmath}
\usepackage{geometry}
\usepackage{tikz}
\usepackage{tikz-cd}
\usepackage[varg]{txfonts}
\usepackage{graphicx}
\usepackage{bm}
\usepackage{enumitem}
\usepackage{indentfirst}
\usepackage{comment}
\usepackage{ulem}
\usepackage{color}
\usepackage{cases}

\newtheorem{thm}{Theorem}[section]
\newtheorem{rem}[thm]{Remark}
\newtheorem{ex}[thm]{Example}

\newtheorem{prop}[thm]{Proposition}
\newtheorem{nota}[thm]{Notation}

\newtheorem{cor}[thm]{Corollary}
\newtheorem{clm}[thm]{Claim}

\theoremstyle{definition}
\newtheorem{defi}[thm]{Definition}

\renewcommand{\Bbb}{\mathbb}
\def\NZQ{\Bbb}

\def\CC{{\NZQ C}}

\def\II{{\NZQ I}}

\def\PP{{\NZQ P}}

\def\RR{{\NZQ R}}
\def\SS{{\NZQ S}}
\def\TT{{\NZQ T}}

\def\ZZ{{\NZQ Z}}

\def\A{{\mathcal A}}
\def\B{{\mathcal B}}

\newcommand{\Zt}[0]{\mathcal Z}

\title{A classification of combinatorial types of discriminantal arrangements}
\author{So Yamagata}
\address{Department of Mathematics,
Hokkaido University, Japan.}
\email{so.yamagata@math.sci.hokudai.ac.jp}

\begin{document}

\begin{abstract}
    Manin and Schechtman introduced a family of arrangements of hyperplanes generalizing classical braid arrangements, which they called the {\it discriminantal arrangements}. Athanasiadis proved a conjecture by Bayer and Brandt providing a full description of the combinatorics of discriminantal arrangements in the case of \textit{very generic} arrangements. Libgober and Settepanella described a sufficient geometric condition for given arrangements to be \textit{non-very generic} in terms of the notion of dependency for a certain arrangement. Settepanella and the author generalized the notion of dependency introducing $r$-sets and $K_\TT$-vector sets, and provided a sufficient condition for non-very genericity but still not convenient to verify by hand. In this paper, we give a classification of the $r$-sets, and a more explicit and tractable condition for non-very genericity.
\end{abstract}

\maketitle
%%%%%%%%%%%%%%%%%%%%%%%%%%%%%%%%%%%%%%%%%%%%%%%%%%%%%%%%%%%%%%%%%%%%%%
\section{Introduction}\label{sec:intro}
A codimension one subspace in a vector space is called a \textit{hyperplane}, and a finite set of hyperplanes is called an \textit{arrangement of hyperplanes} or simply \textit{arrangement}. A basic example of an arrangement is the \textit{braid arrangement} $Br(n) = \{ H_{i,j} \}_{1 \leq i<j \leq n}$, where $H_{i,j} = \{ (x_1, \dots, x_n) \in \CC^k \mid x_i = x_j, \ i \neq j \}$. 

In 1989, Manin and Schechtman \cite{MS} introduced the \textit{discriminantal arrangement} as a generalization of the braid arrangement. In some literatures it is also called the \textit{Manin-Schechtman arrangement}. In brief, the arrangement is defined as follows. 
For a fixed generic arrangement $\A^0 = \{ H_1^0, \dots, H_n^0 \}$ consider the space $\SS(\A^0)$ of all parallel translations of hyperplanes in $\A^0$, which is naturally isomorphic to $\CC^n$. The subset of the space, which consists of all translations of hyperplanes failing to be general position is a hyperplane in $\SS(\A^0) \simeq \CC^n$. The set of such hyperplanes is the discriminantal arrangement and denoted by $\B(n,k,\A^0), n,k \in {\bf N}$ for $k \ge 2$. In particular, $\B(n,1)=\B(n,1,\A^0)$ coincides with the classical braid arrangement.

The discriminantal arrangements relate various areas of mathematics such as vanishing of cohomology of bundles on toric varieties (see \cite{Per}), the higher braid groups and higher categorical perspective (see \cite{KV1}, \cite{KV2}, \cite{KV3}, \cite{Law}, \cite{Koh}), and higher Bruhat orders (see \cite{FZ}, \cite{Zie}).

Athanasiadis \cite{Atha} pointed out that Crapo \cite{Crapo} was doing pioneering work in which he introduced the \textit{geometry of circuits}. In \cite{Crapo} he studied the matroid $M(n,k, \mathcal{C})$ of circuits of the configuration $\mathcal{C}$ of $n$ generic points in $\RR^k$. The circuits of the matroid $M(n,k, \mathcal{C})$ are now the hyperplanes of $\B(n,k,\A^0)$, where $\A^0$ is an arrangement of $n$ hyperplanes in $\RR^k$ orthogonal to the vectors joining the origin with the $n$ points in $\mathcal{C}$ (for further development see \cite{CR}).

Both Manin-Schechtman \cite{MS} and Crapo \cite{Crapo} were mainly interested in arrangements $\B(n,k,\A^0)$ for which the intersection lattice is constant when $\A^0$ varies within a Zariski open set $\Zt$ in the space of generic arrangements of $n$ hyperplanes in $k$ dimensional space. Crapo showed that, in this case, the matroid $M(n,k)$ is isomorphic to the Dilworth completion of the $k$-th lower truncation of the Boolean algebra of rank $n$. 

As for combinatorics of $\B(n,k,\A^0)$ for $\A^0$ to be very generic several results were given. In 1997, Bayer and Brandt (see \cite{BB}) conjectured a full description of the combinatorics of $\B(n,k,\A^0)$ when $\A^0$ belongs to $\Zt$, and it is proved by Athanasiadis \cite{Atha} in 1999. Following \cite{Atha} (more precisely Bayer and Brandt), we call arrangements $\A^0$ in $\Zt$ {\it very generic}, and {\it non-very generic} otherwise.

On the other hand, understanding the combinatorics of $\B(n,k,\A^0)$ for $\A^0$ to be non-very generic has still been incomplete. The first example of non-very generic arrangement was given by Crapo \cite{Crapo} in 1985. The arrangement consists of $6$ generic lines in $\RR^2$, which admits translations that are respectively sides and diagonals of a quadrilateral (see Figure \ref{fig:quadline}). 

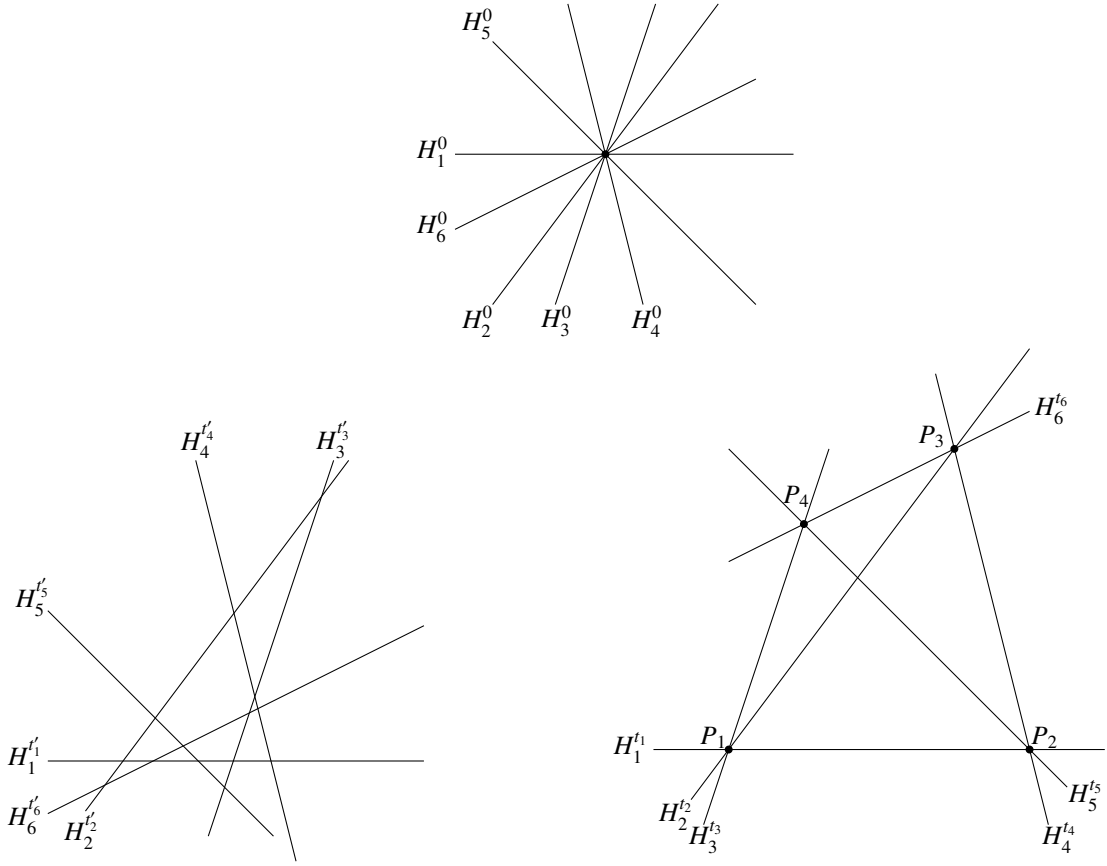
\begin{figure}[h]
    \centering
    \begin{tikzpicture}
        \coordinate (0) at (-2,0);
        \coordinate (1) at (-3/2,-2);
        \coordinate (2) at (-2/3,-2);
        \coordinate (3) at (1/2,-2);
        \coordinate (4) at (2,-2);
        \coordinate (5) at (5/2,0);
        \coordinate (6) at (2,1);
        \coordinate (7) at (3/2,2);
        \coordinate (8) at (-1/2,2);
        \coordinate (9) at (2/3,2);
        \coordinate (10) at (-3/2,3/2);
        \coordinate (12) at (-2,-1);
        
        \coordinate [label=$H_1^0$] (H1) at (-2.3,-0.3);
        \coordinate [label=$H_2^0$] (H2) at (-1.7,-2.55);
        \coordinate [label=$H_3^0$] (H3) at (-0.65,-2.55);
        \coordinate [label=$H_4^0$] (H4) at (0.55,-2.55);
        \coordinate [label=$H_5^0$] (H5) at (-1.7,1.4);
        \coordinate [label=$H_6^0$] (H6) at (-2.3,-1.3);
        \coordinate (O) at (0,0);
        
        \begin{scope}
            \draw (0) -- (5);
            \draw (1) -- (7);
            \draw (3) -- (8);
            \draw (2) -- (9);
            \draw (4) -- (10);
            \draw (6) -- (12);
            \fill (O) circle (1.5pt);
            
        \end{scope}
    \end{tikzpicture}\\
    \begin{minipage}[b]{0.48\columnwidth}
        \begin{tikzpicture}
            \coordinate (0) at (-1,0);
            \coordinate (1) at (-1/2,-2/3);
            \coordinate (2) at (4/3 - 1/5,-1);
            \coordinate (3) at (2/3+0.5 - 1/5,4);*
            \coordinate (4) at (-1,2);*
            \coordinate (5) at (4,0);
            \coordinate (6) at (-1,-1/2-0.2);*
            \coordinate (7) at (3,4);
            \coordinate (8) at (2 - 1/5+0.5,-4/3);*
            \coordinate (9) at (3 - 1/5,4);
            \coordinate (10) at (2,-1);*
            \coordinate (12) at (4,2-0.2);*
            \coordinate [label=$H_1^{t'_1}$] (H1) at (-1.3,-0.3);
            \coordinate [label=$H_2^{t'_2}$] (H2) at (-0.55,-1.3);
            \coordinate [label=$H_3^{t'_3}$] (H3) at (2.8,3.9);
            \coordinate [label=$H_4^{t'_4}$] (H4) at (0.5+0.5,3.9);
            \coordinate [label=$H_5^{t'_5}$] (H5) at (-1.2,1.3+0.5);
            \coordinate [label=$H_6^{t'_6}$] (H6) at (-1.3,-1.1);
            \begin{scope}
                \draw (0) -- (5); 
                \draw (1) -- (7); 
                \draw (3) -- (8); 
                \draw (2) -- (9); 
                \draw (4) -- (10); 
                \draw (6) -- (12); 
            \end{scope}
        \end{tikzpicture}
    \end{minipage}
    \begin{minipage}[b]{0.48\columnwidth}
        \begin{tikzpicture}
            \coordinate (0) at (-1,0);
            \coordinate (1) at (-1/2,-2/3);
            \coordinate (2) at (-1/3,-1);
            \coordinate (3) at (17/4,-1);
            \coordinate (4) at (9/2,-1/2);
            \coordinate (5) at (5,0);
            \coordinate (6) at (4,9/2);
            \coordinate (7) at (4,16/3);
            \coordinate (8) at (11/4,5);
            \coordinate (9) at (4/3,4);
            \coordinate (10) at (0,4);
            \coordinate (12) at (0,5/2);
            \coordinate (p1) at (0,0);
            \coordinate (p2) at (4,0);
            \coordinate (p3) at (3,4);
            \coordinate (p4) at (1,3);
            
            \coordinate [label=$H_1^{t_1}$] (H1) at (-1.3,-0.3);
            \coordinate [label=$H_2^{t_2}$] (H2) at (-0.7,-1.2);
            \coordinate [label=$H_3^{t_3}$] (H3) at (-0.3,-1.5);
            \coordinate [label=$H_4^{t_4}$] (H4) at (4.4,-1.5);
            \coordinate [label=$H_5^{t_5}$] (H5) at (4.75,-0.95);
            \coordinate [label=$H_6^{t_6}$] (H6) at (4.3,4.2);
            \coordinate [label=$P_1$] (P1) at (-0.2,-0.1);
            \coordinate [label=$P_2$] (P2) at (4.2,-0.1);
            \coordinate [label=$P_3$] (P3) at (2.7,3.9);
            \coordinate [label=$P_4$] (P4) at (0.9,3.1);
            
            \begin{scope}
                \draw (0) -- (5);
                \draw (1) -- (7);
                \draw (3) -- (8);
                \draw (2) -- (9);
                \draw (4) -- (10);
                \draw (6) -- (12);
                \fill (p1) circle (1.5pt);
                \fill (p2) circle (1.5pt);
                \fill (p3) circle (1.5pt);
                \fill (p4) circle (1.5pt);
                
            \end{scope}
        \end{tikzpicture}
    \end{minipage}
    \caption{Central generic arrangement of 6 lines in $\RR^2$, its generic translation on the left and its non-(very) generic translation on the right.}\label{fig:quadline}
\end{figure}

However, its non-very genericity did not get much attention at that time. In 1994, after the definition of discriminantal arrangement by Manin-Schechtman, Falk \cite{Falk} constructed an arrangement of hyperplanes spanned by some generic points, and showed that the arrangement $\B(n, k, \A^0 )$ realizes an adjoint of the matroid determined by the points. Using the description he provided a second example of non-very generic arrangement. In this direction, Numata-Takemura \cite{NT} and Koizumi-Numata-Takemura \cite{KNT} gave some computational results on the characteristic polynomials of the discriminantal arrangements.

In 2018, the first general results on non-very generic arrangements were provided. In \cite{LS}, Libgober and Settepanella described a sufficient {\it geometric} condition for the arrangement $\A^0$ to be non-very generic. This condition ensures that $\B(n, k, \A^0)$ admits codimension $2$ strata of multiplicity $3$, which do not exist in the very generic case. It is given in terms of the notion of \textit{dependency} for the arrangement $\A_{\infty}$ in $\PP^{k-1}$ of hyperplanes $H_{\infty,1}, \dots, H_{\infty,n}$, which are the intersections of projective closures of $H_1^0, \dots, H_n^0 \in \A^0$ with the hyperplane at infinity. Their main result shows that $\B(n,k,\A^0), k>1$ admits a codimension 2 stratum of multiplicity $3$ if and only if $\A_{\infty}$ is an arrangement in $\PP^{k-1}$ admitting a restriction which is a dependent arrangement. This construction generalizes Falk's example which corresponds to the case $n=6, k=3$ and which has been object of study in two subsequent papers by Sawada, Settepanella and the author \cite{SSY1}, \cite{SSY2}. 

In 2021, Settepanella and the author \cite{SY} generalized the dependency condition given in \cite{LS}, providing a sufficient condition for the existence of non-very generic intersections in rank $r \geq 2$, i.e., intersections which do not exist in $\B(n,k,\A^0), \A^0 \in \Zt$ in terms of $r$-sets and $K_\TT$-vector sets.
More recently in 2022, Settepanella and the author \cite{SY2} gave a linear condition for non-very genericity as a continuation of \cite{SY}. Some related works have also been continued. In \cite{DPS}, Das-Palezzato-Settepanella  provided examples of the classification of special configurations of points in the $k$-dimensional space relating to the combinatorics of $\B(n,2,\A^0)$ and $\B(n,3,\A^0)$ for $\A^0$ to be non-very generic. In \cite{SS} Saito-Settepanella gave a characterization and a classification of few non-very generic arrangements in low dimensional space. In particular, they classified the combinatorics of $\B(6,3,\A^0)$ over commutative field of characteristic 0. 

Though in \cite{SY} and \cite{SY2}, they constructed conditions for non-very genericity they are still difficult to check by hand. The purpose of this paper is that we classify $r$-sets $\TT$, which are introduced in \cite{SY} and \cite{SY2}, into \textit{non-intersecting type} and \textit{intersecting type}, and then give a sufficient condition to have $K_\TT$-vector sets. This paper is organized as follows. In Section \ref{sec:prem} we recall basic definitions of discriminantal arrangements, $r$-sets and $K_\TT$-vector sets. We also recall a sufficient condition for non-very genericity following \cite{SY2}. In Section \ref{sec:ex} we see examples of constructions of non-very generic arrangements, which sets the stage for the later sections. In Section \ref{sec:classification} we classify $r$-sets into non-intersecting and intersecting types, which are purely combinatorial descriptions. In Section \ref{sec:good_rs} we give a sufficient condition to have the $K_\TT$-vector sets in the case of $r$-sets of non-intersecting type. Moreover we define a special class of $r$-sets of intersecting type, which we call the good $rs$-partition, and give a sufficient condition to have the $K_\TT$-vector sets. This constitutes the first classification of $r$-sets of intersecting type. 
%%%%%%%%%%%%%%%%%%%%%%%%%%%%%%%%%%%%%%%%%%%%%%%%%%%%%%%%%%%%%%%%%%%%%%
%%%%%%%%%%%%%%%%%%%%%%%%%%%%%%%%%%%%%%%%%%%%%%%%%%%%%%%%%%%%%%%%%%%%%%
\section{Preliminaries}\label{sec:prem}
\subsection{A hyperplane arrangement and disriminantal arrangement}
Let us consider an arrangement of hyperplanes in $\CC^k$, i.e., a finite set of hyperplanes in $\CC^k$. For a linear hyperplane $H^0$ define its translation by $H^t = H^0 + \alpha t$, where $\alpha$ is a normal vector to $H^0$, and $t \in \CC$. We denote an arrangement of linear hyperplanes by $\A^0 = \{ H_1^0, \dots, H_n^0 \}$ and its translation by $\A^t = \{ H_1^{t_1}, \dots, H_n^{t_n} \}$, where $H_i^{t_i}$ is a translation of $H_i^0$ throughout this paper. We say that an arrangement of hyperplanes is \textit{generic} if for all $J \subset [n]$, $\lvert J \rvert = k$ normal vectors $\alpha_i$ to $H_i^0$, $i \in J$ are linearly independent. Hyperplanes $H_i^{t_i}$, $i=1, \dots, n$ are said to be \textit{in general position} if the following two conditions are satisfied:
\begin{itemize}
    \item For $1 \leq m \leq k$, the intersection of any $m$ hyperplanes has dimension $k - m$,
    \item For $m > k$, the intersection of any $m$ hyperplanes is empty.
\end{itemize}

Let $\A^0 = \{ H_1^0, \dots, H_n^0 \}$ be a generic arrangement in $\CC^k, k < n$. The space of parallel translations $\SS(\A^0)$ (or simply $\SS$ when dependence on $H_i^0$ is clear or not essential) is the space of $n$-tuples of translations $H_1^{t_1}, \dots, H_n^{t_n}$ such that either $H_i^{t_i} \cap H_i^0 = \emptyset$ or $H_i^{t_i} = H_i^0$ for $i=1, \dots, n$. \\
We can identify $\SS$ with $n$-dimensional affine space $\CC^n$ in such a way that $(H_1^0, \dots, H_n^0)$ corresponds to the origin. In particular, an ordering of hyperplanes in $\A^0$ determines the coordinate system in $\SS$ (see \cite{LS}). 

For a fixed generic arrangement $\A^0$, consider the closed subset of $\SS$ formed by those collections which fail to form a general position. This subset of $\SS$ is a union of hyperplanes $D_L \subset \SS$ (see \cite{MS}). Each hyperplane $D_L$ corresponds to a subset $L = \{ i_1, \dots, i_{k+1} \} \subset [n]$, and it consists of $n$-tuples of translations of hyperplanes $H_1^0, \dots, H_n^0$ in which translations of $H_{i_1}^0, \dots, H_{i_{k+1}}^0$ fail to form a general position. The arrangement $\B(n, k, \A^0)$ of hyperplanes $D_L$ is called $discriminantal$ $arrangement$ and has been introduced by Manin and Schechtman in \cite{MS}. Although they defined the discriminantal arrangement starting from a general position arrangement instead of its centrally translated one, we adopt the latter for convenience.
%%%%%%%%%%%%%%%%%%%%%%%%%%%%%%%%%%%%%%%%%%%%%%%%%%%%%%%%%%%%%%%%%%%%%%
\subsection{(Non) very generic arrangements and $r$-simple intersections}\label{subsec:simple}
It is well known (see among others \cite{Crapo}, \cite{MS}) that there exists an open Zariski set $\Zt$ in the space of generic arrangements of $n$ hyperplanes in $\CC^k$ such that the intersection lattice of the discriminantal arrangement $\B(n,k,\A^0)$ is independent from the choice of the generic arrangement $\A^0 \in \Zt$. Bayer and Brandt \cite{BB} called the arrangements $\A^0 \in \Zt$ \textit{very generic} and the ones $\A^0 \notin \Zt$ \textit{non-very generic}. The name very generic comes from the fact that in this case the cardinality of the intersection lattice of $\B(n,k,\A^0)$ is the largest possible for all generic arrangement of $n$ hyperplanes in $\CC^k$. 

In \cite{Crapo}, Crapo showed that the intersection lattice of $\B(n,k,\A^0)$ for very generic arrangement $\A^0$ is isomorphic to the Dilworth completion $D_k(B_n)$ of a $k$-times lower-truncated Boolean algebra. In \cite{Atha} Athanasiadis gave a more precise description that the intersection lattice of $\B(n,k,\A^0)$ for very generic arrangement $\A^0$ is isomorphic to the lattice $P(n,k)$ defined as follows. $P(n,k)$ is the collection of all sets of the form $\{ S_1, \dots, S_r \}$, where $S_i \subset [n]$, $\lvert S_i \rvert \geq k+1$ such that 
\begin{equation}\label{condi:ineq}
    \lvert \bigcup_{i \in I} S_i \rvert  > k + \sum_{i \in I} (\lvert S_i \rvert - k)
\end{equation}
for all $I \subset [r], \lvert I \rvert \geq 2$. The order on $P(n,k)$ is given by letting $\{ S_1, \dots, S_r \} < \{ T_1, \dots, T_{r'} \}$ if for each $1 \leq i \leq r$ there exists $1 \leq j \leq r'$ such that $S_i \subset T_j$. The isomorphism was first conjectured by Bayer-Brandt in \cite{BB}. 

In \cite{LS}, Libgober-Settepanella gave a full description of rank 2 elements of the intersection lattice of $\B(n,k,\A^0)$. Based on their result, Sawada-Settepanella and the author \cite{SSY1}, \cite{SSY2} showed that hyperplanes in the non-very generic arrangements give rise to special configurations such as the Pappus or Hesse configuration. In \cite{Ku}, Kumar showed that certain line arrangements in the plane always give rise to non-very generic arrangements. In \cite{DPS}, Das-Palezzato-Settepanella  provided examples of the classification of special configurations of points in the $k$-dimensional space relating to the combinatorics of $\B(n,2,\A^0)$ and $\B(n,3,\A^0)$ for $\A^0$ to be non-very generic. In \cite{SS} Saito-Settepanella gave a characterization and a classification of few non-very generic arrangements in low dimensional space. In particular, they classified the combinatorics of $\B(6,3,\A^0)$ over commutative field of characteristic 0.
More recently, Settepanella and the author \cite{SY}, \cite{SY2} provided a geometric and algebraic conditions for arrangement $\A^0$ to be non-very generic, and some examples of non-very generic arrangements. 

In general, it would be complicated to consider intersections $\bigcap_{i=1}^r D_{S_i}$, $\lvert S_i  \rvert > k$ with $\bigcap_{i \in I} D_{L_i} \neq D_s$, where $D_S = \bigcap_{L \subset S, \lvert L \rvert = k+1} D_L$, $D_L \in \B(n,k,\A^0)$ because there would be a lot of case separations on the cardinality of $S_i$. As a first step in \cite{SY} and \cite{SY2} they introduced a \textit{simple} intersection which we call $r$-simple for simplicity.
\begin{defi}\label{def:r-simple}
    An element $X$ in the intersection lattice of the discriminantal arrangement $\B(n,k,\A^0)$ is called an \textit{$r$-simple} if 
    \begin{equation*}
        X = \bigcap_{i=1}^r D_{L_i}, \lvert L_i \rvert = k + 1, 
    \end{equation*}
    and for any $I \varsubsetneq [r]$, $\lvert I \rvert \geq 2$ and any $S \subset [n]$, $\lvert S \rvert > k + 1$,
    it follows that 
    \begin{equation*}
        \bigcap_{i \in I} D_{L_i} \neq D_s, 
    \end{equation*}
    where $D_S = \bigcap_{L \subset S, \lvert L \rvert = k+1} D_L, D_L \in \B(n,k,\A^0)$.
\end{defi}
Though the $r$-simplicity is a central notion, there are no examples of arrangements giving rise to simple and non-simple intersections in the literatures. 
Let us see some examples of arrangements giving rise to 5-simple and non 5-simple intersections. 
\begin{ex}[5-simple intersection]\label{ex:5simple}
    Let $L_1 = \{ 1,2,3,4 \}, L_2 = \{ 1,5,6,7 \}, L_3 = \{ 2,5,8,9 \}, L_4 = \{ 3,6,8,10 \}$ and $L_5 = \{ 4,7,9,10 \}$ be subsets of $[10]$. Let $\A^0$ be a generic arrangement of 10 hyperplanes (planes) in $\CC^3$ and $\A^{t}$ be its translated one as in Figure \ref{fig:simplenonve}. The intersection $X = \bigcap_{i=1}^5 D_{L_i}$ consists of all translations of hyperplanes $H_i^0$ in such a way that $\bigcap_{j=1}^4 D_{L_{i_j}} = \bigcap_{i=1}^5 D_{L_i}$. It also satisfies that for any $I \varsubsetneq [5]$, $\lvert I \rvert \geq 2$ and any $S \subset [10]$, $\lvert S \rvert > 4$, $\bigcap_{i \in I} D_{L_i} \neq D_s$, where $D_S = \bigcap_{L \subset S, \lvert L \rvert = 4} D_L$, $D_L \in \B(10,3,\A^0)$. Thus, $X$ is a 5-simple intersection. \\
    Notice that since there is a correspondence
    \begin{equation*}
        \A^t \in D_{L_i} \Leftrightarrow P_i^t = \bigcap_{p \in L_i} H_p^{t_p} \neq \emptyset,
    \end{equation*}
    the relation is equivalent to saying that if 
    \begin{equation*}
        P_{i_j}^t = \bigcap_{p \in L_{i_j}} H_p^{t_p} \neq \emptyset, \quad j=1,2,3,4,   
    \end{equation*}
    then 
    \begin{equation*}
        P_{i_5}^t = \bigcap_{p \in L_{i_5}} H_p^{t_p} \neq \emptyset
    \end{equation*}
    for any $i_1, i_2, i_3, i_4, i_5 \in [5]$. Each $P_{i_j}^t$ is an intersection point of exactly four hyperplanes $H_p^{t_p}$, $p \in L_{i_j}$ (see Figure \ref{fig:simplenonve}). 
    \begin{figure}[h]
        \centering
        \begin{tikzpicture}
            \coordinate (0) at (-1, 0);
            \coordinate (1) at (-1/2, -5/12);
            \coordinate (2) at (-1/4,-1);
            \coordinate (3) at (1/3,-5/6);
            \coordinate (4) at (5/3,-5/6);
            \coordinate (5) at (22/10, -4/5);
            \coordinate (6) at (27/10, -7/12);
            \coordinate (7) at (3,0);
            \coordinate (8) at (7/2,17/8);
            \coordinate (9) at (7/2,5/2);
            \coordinate (10) at (7/2, 35/12);
            \coordinate (11) at (33/10, 13/4);
            \coordinate (12) at (3/2, 35/8);
            \coordinate (13) at (12/10, 24/5);
            \coordinate (14) at (8/10, 24/5);
            \coordinate (15) at (1/2, 35/8);
            \coordinate (16) at (-3/2,15/4);
            \coordinate (17) at (-3/2,35/12);
            \coordinate (18) at (-3/2, 5/2);
            \coordinate (19) at (-3/2,17/8);

            \coordinate [label=left:$P_1^t$] (a) at (-0.05,0.25);
            \coordinate [label=right:$P_2^t$] (b) at (2.05,0.25);
            \coordinate [label=$P_3^t$] (c) at (2.9,5/2);
            \coordinate [label=right:$P_4^t$] (d) at (1,4);
            \coordinate [label=$P_5^t$] (e) at (-1,5/2);
            
            \coordinate (A) at (0,0);
            \coordinate (B) at (2,0);
            \coordinate (C) at (3,5/2);
            \coordinate (D) at (1,4);
            \coordinate (E) at (-1,5/2);

            \coordinate [label=$H_1^t$] (p) at (-1.3,-0.3);
            \coordinate [label=$H_2^t$] (q) at (-2/3-0.1,-0.95);
            \coordinate [label=$H_3^t$] (r) at (-1/4-0.1,-1.6);
            \coordinate [label=$H_4^t$] (s) at (0.45, -1.4);
            \coordinate [label=$H_5^t$] (t) at (1.7, -1.4);
            \coordinate [label=$H_6^t$] (u) at (2.3, -1.4);
            \coordinate [label=$H_7^t$] (v) at (2.9, -1.1);
            \coordinate [label=$H_8^t$] (w) at (3.8, 1.6);
            \coordinate [label=$H_9^t$] (w) at (3.8, 2.2);
            \coordinate [label=$H_{10}^t$] (t) at (1.85, 4.15);
            \coordinate (x) at (4.8, 2.2);
            \coordinate (y) at (2.7, 4.3);
            
            \begin{scope}
                \draw (0) -- (7);
                \draw (1) -- (10);
                \draw (2) -- (13);
                \draw (3) -- (16);
                \draw (4) -- (11);
                \draw (5) -- (14);
                \draw (6) -- (17);
                \draw (8) -- (15);
                \draw (9) -- (18);
                \draw (12) -- (19);
                
                \fill (A) circle (1.5pt);
                \fill (B) circle (1.5pt);
                \fill (C) circle (1.5pt);
                \fill (D) circle (1.5pt);
                \fill (E) circle (1.5pt);
            \end{scope}
        \end{tikzpicture}
        \caption{An arrangement $\A^t \in X = \bigcap_{i=1}^5 D_{L_i}$.}\label{fig:simplenonve}
    \end{figure}
\end{ex}

\begin{ex}[Non 5-simple intersection]
    Let $L_1 = \{ 1,2,3 \}, L_2 = \{ 1,2,5 \}, L_3 = \{ 1,4,7 \}$, $L_4 = \{ 3,6,7 \}$ and $L_5 = \{ 4,5,6 \}$ be subsets of $[7]$. Let $\A^0$ be a generic arrangement of $7$ hyperplanes (lines) in $\CC^2$ and let $\A^t$ be its translated one, as shown in Figure \ref{fig:nonsimplenonv}. Since there exists an intersection $X = D_{L_1} \cap D_{L_2} = D_{\{ 1,2,3,5 \}}$, the element $X = \bigcap_{i=1}^5 D_{L_i}$ is not a 5-simple intersection. Notice that the ``multiple intersections'' $P_1^t = P_2^t$ are intersections of not three but four hyperplanes, while $P_3^t$, $P_4^t$, $P_5^t$ are intersections of exactly three hyperplanes.
    \begin{figure}[h]
        \centering
        \begin{tikzpicture}
            \coordinate [label=above:${P_1^t=P_2^t}$] (0) at (0,0);
            \coordinate [label=above:$P_3^t$] (1) at (2,1);
            \coordinate [label=right:$P_4^t$] (2) at (0,3);
            \coordinate [label=above:$P_5^t$] (23) at (1,0);
            
            \coordinate (3) at (-1,0);    
            \coordinate (4) at (4,0);
            
            \coordinate (5) at (-1,-1/2);
            \coordinate (6) at (4,2);
            
            \coordinate (7) at (-1/2,-1);
            \coordinate (8) at (2,4);
            
            \coordinate (9) at (0,-1.5);
            \coordinate (10) at (0,4);
            
            \coordinate (11) at (4/3,-1);
            \coordinate (12) at (-1/4,15/4);
            
            \coordinate (13) at (4,-1);
            \coordinate (14) at (-1/3,10/3);
            
            \coordinate (15) at (-1/4,-5/4);
            \coordinate (16) at (4,3);
            
            \coordinate [label=$H_5^{t_5}$] (e) at (-1.3,-0.3);
            \coordinate [label=left:$H_1^{t_1}$] (f) at (-0.9,-1/2);
            \coordinate [label=below:$H_2^{t_2}$] (g) at (-1/2-0.1,-1+0.1);
            \coordinate [label=below:$H_3^{t_3}$] (g) at (0,-1.4);
            \coordinate [label=below:$H_6^{t_6}$] (h) at (1.4,-0.9);
            \coordinate [label=below:$H_7^{t_7}$] (j) at (4.2,-0.9);
            \coordinate [label=right:$H_4^{t_4}$] (i) at (3.9,3+0.2);
            
            \begin{scope}
                \draw (3) -- (4);
                \draw (5) -- (6);
                \draw (7) -- (8);
                \draw (9) -- (10);
                \draw (11) -- (12);
                \draw (13) -- (14);
                \draw (15) -- (16);
                
                \fill (0) circle (1.5pt);
                \fill (1) circle (1.5pt);
                \fill (2) circle (1.5pt);
                \fill (23) circle (1.5pt);
            \end{scope}
        \end{tikzpicture}
        \caption{An arrangement $\A^t \in X = \bigcap_{i=1}^5 D_{L_i}$.}\label{fig:nonsimplenonv}
    \end{figure}
\end{ex}
In the rest of this paper we will focus on non-very generic arrangements such that $X=\bigcap_{i=1}^r D_{L_i}$ is an $r$-simple intersection for simplicity (so we have exactly $r$ intersection points $P_i^t$, $i = 1, \dots, r$ which are intersections of $k+1$ hyperplanes indexed in $L_i$, $i = 1, \dots, r$ in the translated arrangement $\A^t \in \bigcap_{i=1}^r D_{L_i}$). 

We call the number $r$ the \textit{multiplicity} of $X$. If $\A^0$ is very generic and satisfies the condition (\ref{condi:ineq}), then it follows that the subspaces $D_{L_i}$, $i = 1, \dots, r$ intersect transversely (see Theorem 3.1 in \cite{Atha}). The fact is equivalent to saying that since $\text{rank} D_{L_i} = 1$,
\begin{equation}
    \text{rank} \bigcap_{i=1}^r D_{L_i} = \sum_{i=1}^r (\lvert L_i \rvert - k) = r.
\end{equation}
Thus, if the intersection lattice of the discriminantal arrangement $\B(n, k, \A^0)$ contains an $r$-simple intersection of rank strictly lower than $r$, that is a multiplicity of $X$, then $\A^0$ is non-very generic (see \cite{SY} for details). 
%%%%%%%%%%%%%%%%%%%%%%%%%%%%%%%%%%%%%%%%%%%%%%%%%%%%%%%%%%%%%%%%%%%%%%
\subsection{A sufficient condition for non-very genericity}\label{subsec:r-simple}
Following \cite{SY} and \cite{SY2} let us recall a sufficient condition for arrangement $\A^0$ to be non-very generic in this subsection. 

For a fixed set $\TT = \{ L_1, \dots, L_r \}$ of subsets $L_i \subset [n]$, $\lvert L_i \rvert = k+1$ and any translated arrangement $\A^t = \{H_1^{t_1}, \dots, H_n^{t_n}\}$ of $\A^0$ we denote $P_i^t = \bigcap_{p \in L_i} H_p^{t_p}$ and $H_{i,j} = \bigcap_{p \in L_i \cap L_j} H_p^{t_p}$. Notice that $P_i^t$ is a point if and only if $\A^t \in D_{L_i}$; it is empty otherwise. 

A non-very generic arrangement $\A^0$ holds the property that if $\A^t \in \bigcap_{j=1}^{r-1} D_{L_{i_j}}$, then $\A^t \in \bigcap_{i=1}^{r} D_{L_{i}}$. In other words, if we translate hyperplanes of $\A^0$ in such a way that $r-1$ intersection points $P_{i_j}^t$, $j=1, \dots, r-1$ appear, then the $r$-th intersection point $P_{i_r}^t$ also appears automatically. To realize such property Settepanella and the author \cite{SY} introduced the following definitions.

\begin{defi}[$r$-set, \cite{SY}]\label{def:r-set}
    If $\TT = \{ L_1, \dots, L_r \}$ satisfies the conditions
    \begin{equation}\label{eq:proper1}
        \bigcup_{i=1}^r L_i = \bigcup_{i \in I \subset [r], \mid I \mid=r-1} L_i  \quad \mbox{ and} \quad L_i \cap L_j \neq \emptyset
    \end{equation}
    for any subset $I \subset [r], \lvert I \rvert = r-1$ and any two indices $1 \leq i < j \leq r$, we call the set $\TT$ an $r$-\textit{set}.
\end{defi}

\begin{defi}($K_\TT$-translated, $K_\TT$-configuration, \cite{SY})
    A translation $\A^t = \{ H_1^{t_1}, \dots, H_n^{t_n} \}$ of $\A^0$ is called \textit{$K_\TT$-translated} if $P_i^{t} = \bigcap_{p \in L_i} H_p^{t_p} \neq \emptyset$ is the intersection of exactly $k+1$ hyperplanes indexed in $L_i$ for any $L_i \in \TT$. \\
    For a $K_\TT$-translation $\A^t$ we call the complete graph having the points $P_i^t$, $i=1, \dots, r$ as vertices and vectors $P_i^tP_j^t \in \bigcap_{p \in L_i \cap L_j} H_p^t$, $1 \leq i<j \leq r$ as edges \textit{$K_\TT$-configuration} and denote by $K_\TT(\A^0)$.
\end{defi}
In \cite{SY2}, a linear condition for non-very genericity was provided in terms of the above definitions. Let us briefly trace the sketch.

Consider essentialization of the discriminantal arrangement ess$(\B(n,k,\A^0))$ in $\CC^{n-k} \simeq \SS / D_{[n]}$. Then, $\A^t \in$ ess$(\B(n,k,\A^0))$ uniquely corresponds to a translation $t \in \CC^n / C \simeq \CC^{n-k}$, where $C = \{ t \in \CC^n \mid \A^t \ \text{is central} \}$ and the following proposition holds.
\begin{prop}[Proposition 3.1, \cite{SY2}]
    Let $\A^0$ be a generic arrangement of $n$ hyperplanes in $\CC^k$. Translations $\A^{t_1}, \dots, \A^{t_d}$ are linearly independent vectors in $\SS/D_{[n]} \simeq \CC^{n-k}$ if and only if $t_1, \dots, t_d$ are linearly independent vectors in $\CC^n/C \simeq \CC^{n-k}$.
\end{prop}
Let $\A^t$ be $K_\TT$-translated and $P_i^t$ be the intersection $\bigcap_{p \in L_i} H_p^{t_p}$. Then, we can consider a unique family $\{ v_{i,j}^t \}$ of vectors $v_{i,j}^t \in \bigcap_{p \in L_i \cap L_j} H_p^t$ such that $P_i^t + v_{i,j}^t = P_j^t$, and We have the following definition.
\begin{defi}[Definition 3.1, \cite{SY2}]
    Let $\TT$ be an $r$-set and $\A^t$ be $K_\TT$-translated of generic arrangement $\A^0$. Fix a number $i_0 \in [r]$. We call the set of vectors $\{ v_{i_0,j}^t \}_{j \neq i_0}$ satisfying $P_{i_0}^t + v_{i_0,j}^t = P_j^t$ for any $j(\neq i_0) \in [r]$ the \textit{$K_\TT$-vector sets}.
\end{defi}

\begin{rem}
    Since we have $v_{i,j}^t = v_{1,j}^t - v_{1,i}^t$ by definition, the set $\{ v_{i_0,j}^t \}_{j \neq i_0}$ is determined by its subset $\{ v_{1,i}^t \}_{i \neq 1}$. We consider the set $\{ v_{1,i}^t \}_{i \neq 1}$ instead of $\{ v_{i,j}^t \}_{j \neq i}$ in the rest of this paper.
\end{rem}
\begin{center}
    \begin{figure}[h]
        \begin{tikzpicture}
            \coordinate [label=above:$P_i^t$] (0) at (0,{3*2/3});
            \coordinate [label=left:$P_{i+1}^t$] (1) at ({-3/sqrt(2)*2/3},{3/sqrt(2)*2/3});
            \coordinate [label=left:$P_{i+2}^t$] (2) at ({-3*2/3},0);
            \coordinate [label=below:$\dots$] (3) at ({-3*2/3},{-1*2/3});
            \coordinate [label=left:$P_{j-1}^t$] (4) at ({-3/sqrt(2)*2/3},{-3/sqrt(2)*2/3});
            \coordinate [label=below:$P_j^t$] (5) at (0,{-3*2/3});
            \coordinate [label=right:$P_{j+1}^t$] (6) at ({3/sqrt(2)*2/3},{-3/sqrt(2)*2/3});
            \coordinate [label=below:$\dots$] (7) at ({3*2/3},{-1*2/3});
            \coordinate [label=right:$P_{i-2}^t$] (8) at (3*2/3,0);
            \coordinate [label=right:$P_{i-1}^t$] (9) at ({3/sqrt(2)*2/3},{3/sqrt(2)*2/3});
            
            \begin{scope}
                \draw[-latex] (0) -- node[above] {$v_{i,i+1}^t$} (1);
                \draw[-latex] (0) -- node[] {$v_{i,i+2}^t$} (2);
                \draw[-latex] (0) -- node[] {$v_{i,j-1}^t$} (4);
                \draw[-latex] (0) -- node[] {$v_{i,j+1}^t$} (6);
                \draw[-latex] (0) -- node[] {$v_{i,i-2}^t$} (8);
                \draw[-latex] (1) -- node[left] {$v_{i+1,i+2}^t$} (2);
                \draw[-latex] (4) -- node[below] {$v_{j-1,j}^t$} (5);
                \draw[-latex] (5) -- node[below] {$v_{j,j+1}^t$} (6);
                \draw[-latex] (8) -- node[right] {$v_{i-2,i-1}^t$} (9);
                \draw[-latex] (0) -- node[above] {$v_{i,i-1}^t$} (9);
                \draw[-latex] (0) -- node[] {$v_{i,j}^t$} (5);
                
                \fill (0) circle (1.5pt);
                \fill (1) circle (1.5pt);
                \fill (2) circle (1.5pt);
                \fill (4) circle (1.5pt);
                \fill (5) circle (1.5pt);
                \fill (6) circle (1.5pt);
                \fill (8) circle (1.5pt);
                \fill (9) circle (1.5pt);

            \end{scope}
        \end{tikzpicture} \ \ \ 
        \begin{tikzpicture}
            \coordinate [label=above:$P_i^t$] (0) at (0,3*2/3);
            \coordinate [label=above:$P_{i+1}^t$] (1) at ({-3/sqrt(2)*2/3},{3/sqrt(2)*2/3});
            \coordinate [label=left:$P_{i+2}^t$] (2) at (-3*2/3,0);
            \coordinate [label=below:$\dots$] (3) at (-3*2/3,-1*2/3);
            \coordinate [label=left:$P_{j-1}^t$] (4) at ({-3/sqrt(2)*2/3},{-3/sqrt(2)*2/3});
            \coordinate [label=below:$P_j^t$] (5) at (0,-3*2/3);
            \coordinate [label=right:$P_{j+1}^t$] (6) at ({3/sqrt(2)*2/3},{-3/sqrt(2)*2/3});
            \coordinate [label=below:$\dots$] (7) at (3*2/3,-1*2/3);
            \coordinate [label=right:$P_{i-2}^t$] (8) at (3*2/3,0);
            \coordinate [label=above:$P_{i-1}^t$] (9) at ({3/sqrt(2)*2/3},{3/sqrt(2)*2/3});
            
            \begin{scope}
                \draw[-latex] (0) -- node[above] {$v_{i,i+1}^t$} (1);
                \draw[-latex] (0) -- node[] {$v_{i,i+2}^t$} (2);
                \draw[-latex] (0) -- node[] {$v_{i,j-1}^t$} (4);
                \draw[-latex] (0) -- node[] {$v_{i,j+1}^t$} (6);
                \draw[-latex] (0) -- node[] {$v_{i,i-2}^t$} (8);
                \draw[-latex] (0) -- node[above] {$v_{i,i-1}^t$} (9);
                \draw[-latex] (0) -- node[] {$v_{i,j}^t$} (5);
                
                \fill (0) circle (1.5pt);
                \fill (1) circle (1.5pt);
                \fill (2) circle (1.5pt);
                \fill (4) circle (1.5pt);
                \fill (5) circle (1.5pt);
                \fill (6) circle (1.5pt);
                \fill (8) circle (1.5pt);
                \fill (9) circle (1.5pt);
            \end{scope}
        \end{tikzpicture}
        \caption{$K_\TT$-configuration $K_\TT(\A^0)$ and its associated $K_\TT$-vector set.}\label{fig:K_T_vect}
    \end{figure}
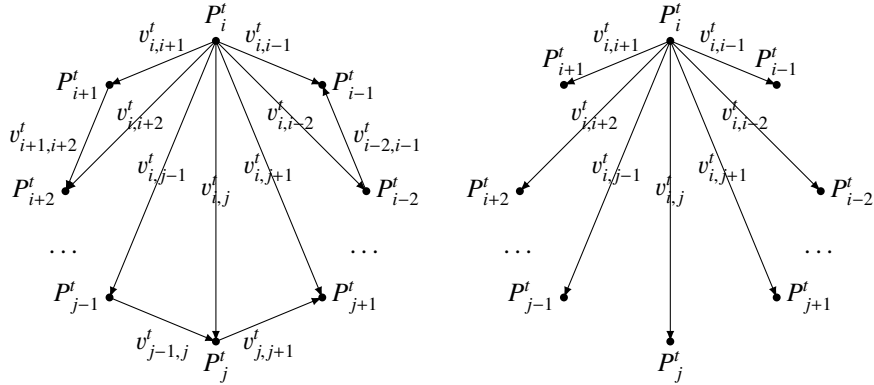
\end{center}
For (a) given $K_\TT$-vector set(s) we define two operations as follows.
\begin{align*}
    \{ v_{1,i}^{t} \}_{i = 2, \dots, r} + \{ v_{1,i}^{t'} \}_{i = 2, \dots, r} & \coloneqq \{ v_{1,i}^{t} + v_{1,i}^{t'} \}_{i = 2, \dots, r} \quad \quad (sum),    \\
    a \{ v_{1,i}^{t} \}_{i = 2, \dots, r}                                      & \coloneqq \{ a v_{1,i}^{t} \}_{i = 2, \dots, r}, a \in \CC \quad (multiplication).
\end{align*}
With above notations and operations, we have the following definition. 
\begin{defi}[Definition 4.2 \cite{SY2}]
    For a fixed $r$-set $\TT$ we call the $d$ different $K_\TT$-vector sets $\{ v_{1,i}^{t} \}_{i = 2, \dots, r}$, $t = 1, \dots, d$ \textit{linearly independent} if for any $a_1, \ldots, a_d \in \CC$ such that
    \begin{equation}
        \sum_{t=1}^{d} a_h \{ v^{t}_{1,i} \}_{i = 2, \dots, r} = 0,
    \end{equation}
    then $a_1=\ldots=a_d=0$. 
\end{defi}
\noindent With the notion of independent $K_\TT$-vector sets the criterion for non-very genericity is provided in \cite{SY2}. The following theorem is a basic result for the rest of this paper.
\begin{thm}[Theorem 4.5 \cite{SY2}]\label{thm:main2}
    Let $\A^0$ be a generic arrangement of $n$ hyperplanes in $\CC^k$. If there exists an $r$-set $\TT=\{L_1, \ldots, L_r\}$ with $\lvert \bigcup_{i=1}^r L_i \rvert = m$ and rank $\bigcap_{p \in \bigcap_{i=1}^r L_i} H_p^0=y$, which admits $m-y-k-r'$ independent $K_{\TT}$-vector sets for some $r'<r$, then $\A^0$ is non-very generic. 
\end{thm}
\begin{rem}\label{rem:vec_set}
    According to Theorem \ref{thm:main2}, if we find a certain number, say $d \in \ZZ_{\geq 1}$ independent $K_\TT$-vector sets, they give rise to a non-very generic arrangement. Indeed we can define hyperplanes $H_l^0 \in \A^0 = \{ H_i^0 \}_{i = 1, \dots, n}$, $l \in L_i \cap L_j$ by $v_{i,j}^t \in H_l^0$, $t \in [d]$.
\end{rem}
\noindent We close this section by giving a notation we will use throughout this paper.
\begin{nota}
    For vectors $v_1, \dots, v_m \in \CC^k$ we denote by $\left< v_{i_1}, \dots, v_{i_k} \right>$ a subspace spanned by $v_{i_1}, \dots, v_{i_k}$. Notice that the vectors $v_{i_1}, \dots,v_{i_k}$ need not necessarily be independent in this notation.
\end{nota}
%%%%%%%%%%%%%%%%%%%%%%%%%%%%%%%%%%%%%%%%%%%%%%%%%%%%%%%%%%%%%%%%%%%%%%
\section{Motivating examples}\label{sec:ex}
Let us begin with Crapo's example.
\begin{ex}[Crapo's example \cite{Crapo}]\label{ex:MS(6,2)}
    Let $\TT = \{ L_1, L_2, L_3, L_4 \}$ be a 4-set defined by $L_1 = \{ 1,2,3 \}, L_2 = \{ 1,4,5 \}, L_3 = \{ 2,4,6 \}, L_4 = \{ 3,5,6 \}$. Consider an arrangement $\A^0 = \{ H_i^0 \}_{i = 1, \dots, 6}$ of lines in $\CC^2$ which admits a $K_\TT$-translation $\A^t$ as Figure \ref{fig:crapo}.  
    \begin{figure}[h]
        \begin{center}
            \begin{tikzpicture}
                \coordinate (0) at (-1,0);
                \coordinate (1) at (-1/2,-2/3);
                \coordinate (2) at (-1/3,-1);
                \coordinate (3) at (17/4,-1);
                \coordinate (4) at (9/2,-1/2);
                \coordinate (5) at (5,0);
                \coordinate (6) at (4,9/2);
                \coordinate (7) at (4,16/3);
                \coordinate (8) at (11/4,5);
                \coordinate (9) at (4/3,4);
                \coordinate (10) at (0,4);
                \coordinate (12) at (0,5/2);
                \coordinate (p1) at (0,0);
                \coordinate (p2) at (4,0);
                \coordinate (p3) at (3,4);
                \coordinate (p4) at (1,3);
                
                \coordinate [label=$H_1^{t_1}$] (H1) at (-1.3,-0.3);
                \coordinate [label=$H_2^{t_2}$] (H2) at (-0.8,-1.0);
                \coordinate [label=$H_3^{t_3}$] (H3) at (-0.3,-1.5);
                \coordinate [label=$H_4^{t_4}$] (H4) at (4.3,-1.5);
                \coordinate [label=$H_5^{t_5}$] (H5) at (4.8,-0.9);
                \coordinate [label=$H_6^{t_6}$] (H6) at (4.3,4.3);
                \coordinate [label=$P_1^t$] (P1) at (-0.2,0);
                \coordinate [label=$P_2^t$] (P2) at (4.2,0);
                \coordinate [label=$P_3^t$] (P3) at (2.7,3.9);
                \coordinate [label=$P_4^t$] (P4) at (0.9,3.1);
                \coordinate [label=$v_{1,3}^t$] () at (1.3,0.8);
                \coordinate [label=$v_{2,4}^t$] () at (2.5,0.8);
                \coordinate (A) at (-0.2,-0.1);
                
                \begin{scope}
                    \draw (0) -- node[below] {$v_{1,2}^t$} (5);
                    \draw (1) -- node[] {} (7);
                    \draw (3) -- node[right] {$v_{2,3}^t$} (8);
                    \draw (2) -- node[left] {$v_{1,4}^t$} (9);
                    \draw (4) -- node[] {} (10);
                    \draw (6) -- node[above] {$v_{3,4}^t$} (12);
                    \draw[-latex] (p1) -- (p2);
                    \draw[-latex] (p1) -- (p3);
                    \draw[-latex] (p1) -- (p4);
                    \draw[-latex] (p2) -- (p3);
                    \draw[-latex] (p2) -- (p4);
                    \draw[-latex] (p3) -- (p4);
                    
                    \fill (p1) circle (1.5pt);
                    \fill (p2) circle (1.5pt);
                    \fill (p3) circle (1.5pt);
                    \fill (p4) circle (1.5pt); 
                    
                \end{scope}
            \end{tikzpicture}\caption{Translation $\A^t$ consisting of four quadrilateral and two diagonal lines.}\label{fig:crapo}
        \end{center}
    \end{figure}
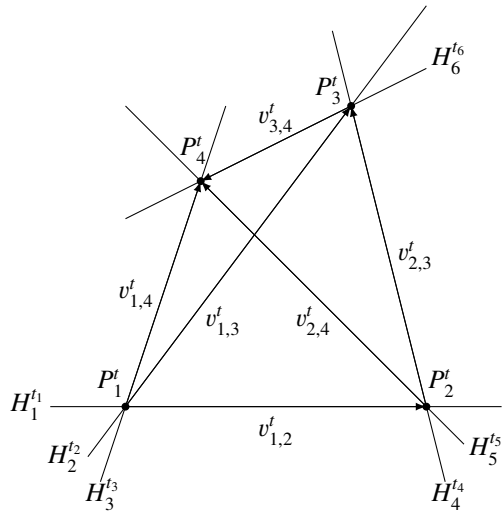 \\
    Let assume the arrangement $\A^0$ admits $d$ $K_\TT$-translations that is, we can choose $d$ linearly independent $K_\TT$-vector sets $\{ v_{1,2}^t, v_{1,3}^t, v_{1,4}^t \}$, $t = 1, \dots, d$. 
    
    For $1 \leq a<b \leq 4$ we denote by
    \begin{align*}
        V_{a,b} = \left< v_{a,b}^t \mid  t = 1, \dots, d \right>
    \end{align*}
    the vector space spanned by $v_{a,b}^t \in H_{a,b}^0 = \bigcap_{p \in L_a \cap L_b} H_p^0$.
    The following two claims hold.
    \begin{clm}\label{clm:1}
        The dimension of $V_{a,b}$ is one for any $a,b$.
    \end{clm}
    \begin{proof}
        Since for any $a,b$ there exists a vector $\alpha_p \in V_{a,b}^\perp$, $p \in L_a \cap L_b$, we have $\dim V_{a,b}^\perp \geq 1$ $\iff$ $\dim V_{a,b} \leq 1$. Since $v_{a,b}^t \neq 0$, and otherwise $H_{a,b}^0 = \{ 0 \}$, we also have $\dim V_{a,b} \neq 0$. Thus, $\dim V_{a,b} = 1$ for any $a,b$.
    \end{proof}
    \begin{clm}\label{clm:2}
        Let $\{ v_{1,2}^t, v_{1,3}^t, v_{1,4}^t \}$, $t = 1, \dots, d$ be $K_\TT$-vector sets. The sets are linearly independent if and only if $d = 1$.
    \end{clm}
    \begin{proof}
        If $d=1$, the set $\{ v_{1,2}^1, v_{1,3}^1, v_{1,4}^1 \}$ is obviously linearly independent. Let us show the converse. If $d > 1$, there should exist a scalar $k$ such that $v_{i,j}^t = k v_{i,j}^1$ for $t = 2, \dots, d$ and any $i, j$, and otherwise $\dim V_{a,b} \geq 2$ for some $a,b \in [4]$. Thus, in this case we have that $\{ v_{1,2}^t, v_{1,3}^t, v_{1,4}^t \} = k \{ v_{1,2}^1, v_{1,3}^1, v_{1,4}^1 \}$ for any $t = 2, \dots, d$; i.e., we have dependent $K_\TT$-vector sets. Thus, if $\{ v_{1,2}^t, v_{1,3}^t, v_{1,4}^t \}$, $t = 1, \dots, d$ are independent $K_\TT$-vector sets, then $d=1$.
    \end{proof}
    By Claim \ref{clm:2} it is sufficient to consider only the case $d=1$. Since each hyperplane (line) $H_p^0$, $p \in L_a \cap L_b$ contains a vector $v_{a,b}$, where $1 \leq a<b \leq 4$ and $\lvert L_a \cap L_b \rvert = 1$, it follows that
    \begin{equation}
        \left< \alpha_{l_{a.b}} \right> = \left( H_{l_{a,b}}^0 \right)^\perp = \sum_{p \in L_a \cap L_b} \left( H_p^0 \right)^\perp = V_{a,b}^\perp,
    \end{equation}
    where $\{ l_{a,b} \} = L_a \cap L_b$. \\
    Notice that the $K_\TT$-vector set $\{ v_{1,2}, v_{1,3}, v_{1,4} \}$ satisfies the condition 
    \begin{equation}\label{eq:1}
        \dim \sum_{a,b \in I} V_{a,b}^\perp = \lvert I \rvert \text{ for any $I \subset [4], 1  \leq \lvert I \rvert \leq 2$}, \ \text{and} \
        v_{i,k} \notin  \left< v_{i,j} \right> \text{for distinct $i,j,k$}.
    \end{equation}
    Conversely, let us consider a set of vectors $\{ v_{1,2}, v_{1,3}, v_{1,4} \}$ satisfying the condition (\ref{eq:1}).
    Then, we can choose generic vectors $\alpha_l$, $l = 1, \dots, 6$ from the orthogonal spaces $V_{a,b}^\perp$, and we obtain a generic arrangement $\A^0$. In particular, the arrangement admits a $K_\TT$-translation. That is we have intersection points $P_i^t = \bigcap_{p \in L_i} H_p^t$, $i=1,2,3,4$, which are intersections of exactly 3 hyperplanes indexed in $L_i$, $i=1,2,3,4$. Thus, we have a $K_\TT$-vector set $\{ v_{1,2}, v_{1,3}, v_{1,4} \}$. By Theorem \ref{thm:main2} it follows that the arrangement $\A^0$ constructed from the $K_\TT$-vector set is non-very generic.
\end{ex}
The following arrangement is constructed as a ``high-dimensional'' Crapo's example in \cite{SY2}.
%%%%
\begin{ex}[$\B(12,8,\A^0)$, \cite{SY2}]\label{ex:MS(12,8)}
    Let $\TT = \{ L_1, L_2, L_3, L_4 \}$ be a 4-set defined by $L_j = [12] \setminus K_j$, where $K_1 = \{ 10,11,12 \}$, $K_2 = \{ 7,8,9 \}$, $K_3 = \{ 4,5,6 \}$, and $K_4 = \{ 1,2,3 \}$. Consider an arrangement $\A^0 = \{ H_i^0 \}_{i = 1, \dots, 12}$ of hyperplanes in $\CC^8$ which admits $d$ $(\geq 1)$ $K_\TT$-translations $\A^t$, $t = 1, \dots, d$ as in Figure \ref{fig:MS(12,8)}. Remark that unlike Example \ref{ex:MS(6,2)}, each line in the figure does not represent a hyperplane but a subspace $H_{i,j}^t = \bigcap_{p \in L_i \cap L_j} H_p^t$ since $\lvert L_i \cap L_j  \rvert = 6$ for all $1 \leq i<j \leq 4$.
    
    Let $\{ v_{1,2}^t, v_{1,3}^t, v_{1,4}^t \}$ be their associated $K_\TT$-vector sets.
    \begin{figure}[h]
        \begin{center}
            \begin{tikzpicture}
                \coordinate (0) at (-1,0);
                \coordinate (1) at (-1/2,-2/3);
                \coordinate (2) at (-1/3,-1);
                \coordinate (3) at (17/4,-1);
                \coordinate (4) at (9/2,-1/2);
                \coordinate (5) at (5,0);
                \coordinate (6) at (4,9/2);
                \coordinate (7) at (4,16/3);
                \coordinate (8) at (11/4,5);
                \coordinate (9) at (4/3,4);
                \coordinate (10) at (0,4);
                \coordinate (12) at (0,5/2);
                \coordinate (p1) at (0,0);
                \coordinate (p2) at (4,0);
                \coordinate (p3) at (3,4);
                \coordinate (p4) at (1,3);
                
                \coordinate [label=$H_{1,2}^{t_1}$] (H1) at (-1.4,-0.3);
                \coordinate [label=$H_{1,3}^{t_2}$] (H2) at (-0.8,-1.1);
                \coordinate [label=$H_{1,4}^{t_3}$] (H3) at (-0.4,-1.6);
                \coordinate [label=$H_{2,3}^{t_4}$] (H4) at (4.3,-1.6);
                \coordinate [label=$H_{2,4}^{t_5}$] (H5) at (4.8,-1);
                \coordinate [label=$H_{3,4}^{t_6}$] (H6) at (4.4,4.2);
                \coordinate [label=$P_1^t$] (P1) at (-0.2,0);
                \coordinate [label=$P_2^t$] (P2) at (4.2,0);
                \coordinate [label=$P_3^t$] (P3) at (2.7,3.9);
                \coordinate [label=$P_4^t$] (P4) at (0.9,3.1);
                \coordinate [label=$v_{1,3}^t$] () at (1.3,0.8);
                \coordinate [label=$v_{2,4}^t$] () at (2.5,0.8);
                
                \begin{scope}
                    \draw (0) -- node[below] {$v_{1,2}^t$} (5);
                    \draw (1) -- node[] {} (7);
                    \draw (3) -- node[right] {$v_{2,3}^t$} (8);
                    \draw (2) -- node[left] {$v_{1,4}^t$} (9);
                    \draw (4) -- node[] {} (10);
                    \draw (6) -- node[above] {$v_{3,4}^t$} (12);
                    \fill (p1) circle (1.5pt);
                    \fill (p2) circle (1.5pt);
                    \fill (p3) circle (1.5pt);
                    \fill (p4) circle (1.5pt); 
                    \draw[-latex] (p1) -- (p2);
                    \draw[-latex] (p1) -- (p3);
                    \draw[-latex] (p1) -- (p4);
                    \draw[-latex] (p2) -- (p3);
                    \draw[-latex] (p2) -- (p4);
                    \draw[-latex] (p3) -- (p4);
                \end{scope}
            \end{tikzpicture}\caption{A $K_\TT$-translated arrangement $\A^t$. Each $H_{i,j}^t$ denotes the interseciton $\bigcap_{p \in L_i \cap L_j} H_p^t$.}\label{fig:MS(12,8)}
        \end{center}
    \end{figure}
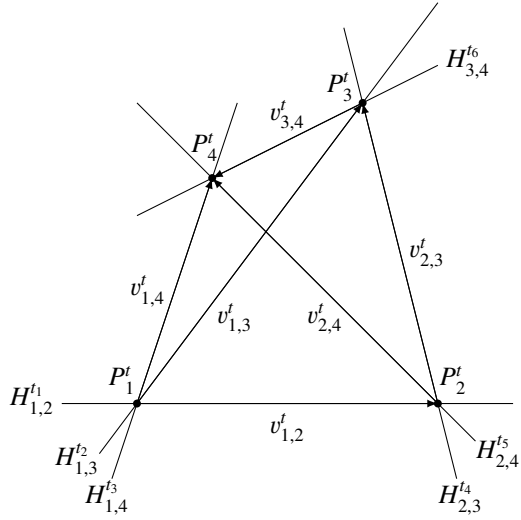
    For $l=1,2,3,4$ denote by
    \begin{align*}
        V_{[4] \setminus \{ l \}} = \left< v_{a,b}^t \mid a,b \in [4] \setminus \{ l \}, t = 1, \dots, d \right>
    \end{align*}
    the vector space spanned by vectors $v_{a,b}^t$, $a,b \in [4] \setminus \{ l \}$, $t = 1, \dots, d$. Notice that we have
    \begin{equation}
        \bigcap_{p \in \bigcap_{t \in [4] \setminus \{ l \}} L_t} H_p^0 \supset V_{[4] \setminus \{ l \}},
    \end{equation}
    equivalently,
    \begin{equation}
        \left< \alpha_p \mid p \in \bigcap_{t \in [4] \setminus \{ l \}} L_t  \right> = \sum_{p \in \bigcap_{t \in [4] \setminus \{ l \}} L_t} \left( H_p^0 \right)^\perp \subset V_{[4] \setminus \left\{ l \right\}}^\perp.
    \end{equation}\\
    Then, the $K_\TT$-vector sets satisfy the condition
    \begin{equation}\label{eq:ex3.1}
        \dim \sum_{l \in I} V_{[4] \setminus \{ l \}}^\perp
        \begin{cases}
            \geq 3 \lvert I \rvert \quad & \text{for any $I \subset [4], 1 \leq \lvert I \rvert \leq 2$}, \text{and} \\
            = 8 \quad                    & \text{for any $I \subset [4], \lvert I \rvert = 3$}.
        \end{cases}
    \end{equation}
    Conversely, let assume there exist $d (\geq 1)$ sets of vectors $\{ v_{1,2}^t, v_{1,3}^t, v_{1,4}^t \}$ satisfying the condition (\ref{eq:ex3.1}). Then, we can choose generic vectors $\alpha_l$, $l = 1, \dots, 12$ from the orthogonal spaces $V_{a,b}^\perp$, and we obtain a generic arrangement $\A^0$. In particular, the arrangement admits $K_\TT$-translations $\A^t$, $t=1, \dots, d$. That is we have intersection points $P_i^t = \bigcap_{p \in L_i} H_p^t$, $i=1,2,3,4$ for each $t = 1, \dots, d$. Thus, the vectors $\{ v_{1,2}^t, v_{1,3}^t, v_{1,4}^t \}$, $t=1, \dots, d$ are $K_\TT$-vector sets. Moreover, if $\TT$ is a 4-set defined as in this example, we can see an explicit construction of the sets $\{ v_{1,2}^t, v_{1,3}^t, v_{1,4}^t \}$, $t = 1, \dots, d$ satisfying (\ref{eq:ex3.1}).
    Let $d_{1,i} = \dim \left< v_{1,i}^t \mid t = 1, \dots, d \right>$, $i = 2,3,4$. The following proposition gives the way how to find the vectors $\{ v_{1,2}^t, v_{1,3}^t, v_{1,4}^t \}$, $t = 1, \dots, d$.
    \begin{prop}\label{lem:mainlemex2}
        The sets $\{ v_{1,2}^t, v_{1,3}^t, v_{1,4}^t \}$, $t = 1, \dots, d$ satisfy (\ref{eq:ex3.1}) if and only if $v_{1,a}^l \in \left<v_{1,a}^t \mid t = 1, \dots, d \setminus \{ l \} \right>$ for any $a \in \{ 2,3,4 \}$, $l \in [d]$ and $d_{1,k} \leq 2$, $d_{1,i} + d_{1,j} \leq 5$ for any $i,j,k$.
    \end{prop}
    \begin{proof}
        First, let us prove that if the sets $\{ v_{1,2}^t, v_{1,3}^t, v_{1,4}^t \}$, $t = 1, \dots, d$ satisfy (\ref{eq:ex3.1}), then $v_{1,a}^l \in \left<v_{1,a}^t \mid t = 1, \dots, d \setminus \{ l \} \right>$ for any $a \in \{ 2,3,4 \}$, $l \in [d]$. We prove this in the case of $a=2$ by contradiction, assuming that there exists a vector $v_{1,2}^j \in \left< v_{1,3}^t, v_{1,4}^t \mid t = 1, \dots, d \right>$. In this case we have $v_{1,2}^j \in V_{[4] \setminus \{ 2 \}}$. \\
        By (\ref{eq:ex3.1}) we have
        \begin{equation}\label{eq:equiv}
            \dim \sum_{l = 2}^4 V_{[4] \setminus \{ l \}}^\perp = 8 \iff \dim \bigcap_{l=2}^4 V_{[4] \setminus \{ l \}} = 0.
        \end{equation}
        On the other hand, since $v_{1,2}^j \in V_{[4] \setminus \{ l \}}$, $l=3,4$ and also $v_{1,2}^j \in V_{[4] \setminus \{ 2 \}}$ by assumption, it follows that $v_{1,2}^j \in \bigcap_{l=2}^4 V_{[4] \setminus \{ l \}}$, which contradicts the fact that $\dim \bigcap_{l=2}^4 V_{[4] \setminus \{ l \}} = 0$. Thus, $v_{1,2}^j \in \left< v_{1,2}^t \mid t = 1, \dots, d \setminus \{ j \} \right>$. The analogous proofs follow for any $a \in \{ 2,3,4 \}$ and $l \in [d]$. \\
        Secondly, let us prove $d_{1,k} \leq 2$ and $d_{1,i} + d_{1,j} \leq 5$ for any $i,j,k$. By the fact we have just proved we obtain
        \begin{equation}\label{eq:dim}
            \begin{split}
                & \left< v_{1,2}^t, v_{1,3}^t \mid t = 1, \dots, d \right> \cap \left< v_{1,4}^t \mid t = 1, \dots, d \right> = \{ 0 \} \ \text{and} \\
                & \left< v_{1,2}^t \mid t = 1, \dots, d \right> \cap \left< v_{1,3}^t \mid t = 1, \dots, d \right> = \{ 0 \}.
            \end{split}
        \end{equation}
        On the other hand, we have
        \begin{equation}\label{eq:equi}
            \displaystyle \dim \sum_{l \in I} V_{[4] \setminus \{ l \}}^\perp \geq 3 \lvert I \rvert \iff \dim \bigcap_{l \in I} V_{[4] \setminus \{ l \}}
            \leq 8 - 3 \lvert I \rvert
        \end{equation}
        for any $I \subset [4], 1 \leq \lvert I \rvert \leq 2$ by (\ref{eq:ex3.1}); thus we have $d_{1,k} \leq 2$ and $d_{1,i} + d_{1,j} \leq 5$ for any $i,j,k$. \\
        Conversely, let us assume $v_{1,a}^l \in \left<v_{1,a}^t \mid t = 1, \dots, d \setminus \{ l \} \right>$ for any $a \in \{ 2,3,4 \}$, $l \in [d]$ and $d_{1,k} \leq 2$, $d_{1,i} + d_{1,j} \leq 5$ for any $i,j,k$. In consideration of (\ref{eq:equi}), if there exist the sets satisfying the assumptions, then the proof would be completed. For this reason it is sufficient to show that there exist such the sets. \\
        Since $v_{1,a}^l \in \left<v_{1,a}^t \mid t = 1, \dots, d \setminus \{ l \} \right>$ for any $a \in \{ 2,3,4 \}$, $l \in [d]$, we have (\ref{eq:dim}). In particular, we have 
        \begin{equation*}
            \dim \left< v_{1,2}^t, v_{1,3}^t, v_{1,4}^t \mid t = 1, \dots, d \right>
        \end{equation*}
        \begin{equation*}
            \begin{split}
                & = \dim \left< v_{1,2}^t \mid t = 1, \dots, d \right> +  \dim \left< v_{1,3}^t \mid t = 1, \dots, d \right> + \dim \left< v_{1,4}^t \mid t = 1, \dots, d \right> \\
                & = d_{1,2} + d_{1,3} + d_{1,4}.
            \end{split}
        \end{equation*}
        To construct the the sets $\{ v_{1,2}^t, v_{1,3}^t, v_{1,4}^t \}$, $t = 1, \dots, d$, we need to choose $3d$ vectors $v_{1,i}^t, i = 2,3,4, t = 1, \dots, d$ with $\dim \left< v_{1,i}^t \mid t = 1, \dots, d \right> = d_{1,i}$, $i = 2,3,4$. In particular, it is sufficient to choose $d_{1,2} + d_{1,3} + d_{1,4}$ independent vectors in $\CC^8$. By assumption we have $d_{1,k} \leq 2$ and $d_{1,i} + d_{1,j} \leq 5$ for any $i,j,k$. Notice that the second inequality automatically follows since the first one holds for any $k$. This implies that $d_{1,2} + d_{1,3} + d_{1,4} \leq 6 < 8$. Thus, the sets we expected actually exist.
    \end{proof}
\end{ex}
%%%%%
%%%%%%%%%%%%%%%%%%%%%%%%%%%%%%%%%%%%%%%%%%%%%%%%%%%%%%%%%%%%%%%%%%%%%%
\section{A classification of $r$-sets}\label{sec:classification}
In this section we classify the $r$-sets $\TT = \{ L_1, \dots, L_r \}$ into non-intersecting and intersecting types. Each $L_i$ is the set of indices of hyperplanes defining a hyperplane $D_{L_i}$ of the discriminantal arrangement. Since we are considering arrangements in $\CC^k$ and focusing on $r$-simple intersections, we assume $\lvert L_i \rvert = k+1$, $i=1, \dots, r$.
\subsection{Non-intersecting type $r$-sets}\label{sec:Benoit_type}
To begin with, let us consider $r$-sets for simple case that intersections of all three sets of an $r$-set are empty. More precisely, we give the following definition.
\begin{defi}
    Let $r \geq 4$. We say that $r$-set $\TT = \{ L_1, \dots, L_r \}$ is \textit{non-intersecting type} if $L_i \cap L_j \cap L_k = \emptyset$ for any distinct $i,j,k$.
\end{defi}
\noindent Let us denote
\begin{equation}
    L_i = \bigcup_{j \in [r] \setminus \{ i \}} A_{i,j},
\end{equation}
where $A_{i,j} = L_i \cap L_j \subset [n]$ and $\lvert A_{i,j} \rvert = a_{i,j}$, $a_{i,j} \geq 1$. \\
Since $A_{1,j} = L_j \setminus \bigcup_{l \in [r] \setminus \{ 1,j \}} A_{l,j}$, we have
\begin{equation}\label{eq:a_ij}
    a_{1,j} = k+1 - \sum_{l \in [r] \setminus \{ 1,j \}} a_{l,j}.
\end{equation}
By summing both sides of the formula (\ref{eq:a_ij}) for $j \in [r] \setminus \{ 1 \}$, we obtain
\begin{equation}
    \begin{split}
        \sum_{j \in [r] \setminus \{ 1 \}} a_{1,j} &= \sum_{j \in [r] \setminus \{ 1 \}} \left( k+1 - \sum_{l \in [r] \setminus \{ 1,j \}} a_{l,j} \right) \\
        &= \sum_{j \in [r] \setminus \{ 1 \}} (k+1) - \sum_{j \in [r] \setminus \{ 1 \}} \sum_{l \in [r] \setminus \{ 1,j \}} a_{l,j}.
    \end{split}
\end{equation}
Thus, we have
\begin{equation}\label{eq:rel}
    k+1 = (r-1)(k+1) - 2\sum_{l,t \in [r] \setminus \{ 1 \}, l<t} a_{l,t}.
\end{equation}
Equivalently,
\begin{equation}
    \sum_{l,t \in [r] \setminus \{ 1 \}, l<t} a_{l,t} = \frac{(r-2)(k+1)}{2}.
\end{equation}
Thus, in terms of $r,k$ the number of hyperplanes can be written as
\begin{equation}\label{eq:allforB}
    \begin{split}
        \lvert \bigcup_{i=1}^r L_i \rvert = \sum_{1 \leq i < j \leq r} a_{i,j} &= \sum_{j \in [r] \setminus \{ 1 \}} a_{1,j} + \sum_{l,t \in [r] \setminus \{ 1 \}, l<t} a_{l,t} \\
        &= (k+1) + \frac{(r-2)(k+1)}{2} = \frac{r(k+1)}{2}.
    \end{split}
\end{equation}
Since $a_{i,j} \geq 1$
\begin{equation}
    \begin{split}
        \sum_{1 \leq i<j \leq r} a_{i,j} &= \sum_{j \in [r] \setminus \{ 1 \}} a_{1,j} + \sum_{l,t \in [r] \setminus \{ 1 \}, l<t} a_{l,t} \\
        & \geq r-1 + \frac{(r-1)(r-2)}{2} = \binom{r}{2}.
    \end{split}
\end{equation}
Notice that since  $\frac{r(k+1)}{2} \geq \binom{r}{2}$, we have
\begin{equation}\label{eq:dimforB}
    k \geq r-2.
\end{equation}

Once we determine the tuple $(a_{i,j})_{i,j}$, the $r$-set $\TT = \{ L_1, \dots, L_r \}$ is uniquely determined up to renumbering of elements in $\displaystyle \bigcup_{i=1}^r L_i$ or indices of the sets $L_i$. Thus, to classify $r$-sets $\TT$ of non-intersecting type, it is enough to determine a tuple $(a_{i,j})_{i,j}$ assuming (\ref{eq:allforB}) and (\ref{eq:dimforB}). In particular, it is enough to determine $a_{l,j}$ for $2 \leq l<j \leq r$, since once $a_{l,j}$ for $2 \leq l <j \leq r$ are determined, the remaining ones $a_{1,j}$, $j \in [r] \setminus \{ 1 \}$ are automatically determined by (\ref{eq:a_ij}). \\
In other words, $r$-sets of non-intersecting type correspond to decompositions into the sum
\begin{equation}
    \frac{(r-2)(k+1)}{2} = \sum_{l,t \in [r] \setminus \{ 1 \}, l<t} a_{l,t}.
\end{equation}
Summarizing the above discussion, we obtain the following proposition.
\begin{prop}\label{prop:miniprop}
    Let $r \geq 4$, $k \geq r-2$ and $\TT = \{ L_1, \dots, L_r \}$ be an $r$-set of non-intersecting type such that $\lvert \bigcup_{i=1}^r L_i \rvert  = \frac{r(k+1)}{2}$. Then, the tuple $(a_{i,j})_{1 \leq i< j \leq r}$ one-to-one corresponds to the tuple $(a_{l,t})_{l,t \in [r] \setminus \{ 1 \}}, l<t$, which also corresponds to the decomposition into the sum
    \begin{equation}
        \frac{(r-2)(k+1)}{2} = \sum_{l,t \in [r] \setminus \{ 1 \}, l<t} a_{l,t}.
    \end{equation}
\end{prop}
Let us see examples of $r$-sets of non-intersecting type.
\begin{ex}[$4$-set of non-intersecting type]
    Let us consider $r = 4$. In this case we have a correspondence between tuples $(a_{i,j})_{2 \leq i<j \leq 4}$ and sum decompositions $\frac{(r-2)(k+1)}{2} = k+1 = a_{2,3} + a_{2,4} + a_{3,4}$, where $k \geq 2$. The following are examples for $k = 2$ and $k = 3$.
    \paragraph{The case $k=2$}
    There exists only one correspondence:
    \begin{equation}
        \begin{split}
            3 = a_{2,3} + a_{2,4} + a_{3,4} &= 1 + 1 + 1 \\
            & \leftrightarrow (a_{1,2}, a_{1,3}, a_{1,4}, a_{2,3}, a_{2,4}, a_{3,4}) = (1,1,1,1,1,1).
        \end{split}
    \end{equation}
    In particular, we have a 4-set $\TT = \{ L_1, L_2, L_3, L_4 \}$ with $ L_1 = \{ 1,2,3 \}, L_2 = \{ 1,4,5 \}, L_3 = \{ 2,4,6 \}, L_4 = \{ 3,5,6 \}$ for example. This 4-set is assumed in Example \ref{ex:MS(6,2)}.
    
    \paragraph{The case $k=3$}
    The following is one example:
    \begin{equation}
        \begin{split}
            4 = a_{2,3}+ a_{2,4}+ a_{3,4} &= 2 + 1 + 1 \\
            & \leftrightarrow (a_{1,2}, a_{1,3}, a_{1,4}, a_{2,3}, a_{2,4}, a_{3,4}) = (1,1,2,2,1,1).
        \end{split}
    \end{equation}
    In particular, we have $\TT = \{ L_1, L_2, L_3, L_4 \}$ with $L_1 = \{ 1,2,3,4 \}, L_2 = \{ 1,5,6,7 \}, L_3 = \{ 2,5,6,8 \}, L_4 = \{ 3,4,7,8 \}$ for example.
\end{ex}

\begin{ex}[$5$-set of non-intersecting type]
    Let us consider $r = 5$. In this case we have correspondence between tuples $(a_{i,j})_{2 \leq i<j \leq 5}$ and sum decompositions $\frac{(r-2)(k+1)}{2} = \frac{3(k+1)}{2} = a_{2,3} + a_{2,4} + a_{2,5} + a_{3,4} + a_{3,5} + a_{4,5}$. Since the number $\frac{3(k+1)}{2}$ corresponds to the number of hyperplanes, $k$ should be odd number with $k \geq 3$. \\
    Let us see examples for $k = 3$ and $k = 5$.
    \paragraph{The case $k=3$}
    There exists only one correspondence:
    \begin{equation}
        \begin{split}
            6 &= a_{2,3} + a_{2,4} + a_{2,5} + a_{3,4} + a_{3,5} + a_{4,5} = 1 + 1 + 1 + 1 + 1 + 1 \\
            &\leftrightarrow (a_{1,2}, a_{1,3}, a_{1,4}, a_{1,5}, a_{2,3},  a_{2,4},  a_{2,5},  a_{3,4},  a_{3,5},  a_{4,5}) = (1,1,1,1,1,1,1,1,1,1).
        \end{split}
    \end{equation}
    For example we have $\TT = \{ L_1, L_2, L_3, L_4, L_5 \}$ with
    $L_1 = \{ 1,2,3,4 \}, L_2 = \{ 1,5,6,7 \}, L_3 = \{ 2,5,8,9 \}, L_4 = \{ 3,6,8,10 \}, L_5 = \{ 4,7,9,10 \}$. This 5-set is provided in Example 5.3 in \cite{SY}.
    
    \paragraph{The case $k=5$}
    Following is one example:
    \begin{align*}
        9 & = a_{2,3} + a_{2,4} + a_{2,5} + a_{3,4} + a_{3,5} + a_{4,5} = 1 + 1 + 1 + 2 + 2 + 2                                                      \\
          & \leftrightarrow (a_{1,2}, a_{1,3}, a_{1,4}, a_{1,5}, a_{2,3},  a_{2,4},  a_{2,5},  a_{3,4},  a_{3,5},  a_{4,5}) = (3,1,1,1,1,1,1,2,2,2).
    \end{align*}
    For example we have $\TT = \{ L_1, L_2, L_3, L_4, L_5 \}$ with $L_1 = \{ 1,2,3,4,5,6 \}, L_2 = \{ 1,2,3,7,8,9 \}, L_3 = \{ 4,7,10,11,12,13 \}, L_4 = \{ 5,8,10,11,14,15 \}, L_5 = \{ 6,9,12,13,14,15 \}$.
\end{ex}
%%%%%%%%%%%%%%%%%%%%%%%%%%%%%%%%%%%%%%%%%%%%%%%%%%%%%%%%%%%%%%%%%%%%%%
\subsection{Intersecting type $r$-sets}
As the next, we consider more complicated case admitting some three sets of an $r$-set are not empty.

\begin{defi}
    Let $r \geq 3$. We say that $r$-set $\TT = \{ L_1, \dots, L_r \}$ is an \textit{intersecting type} if $r=3$ and $L_i \cap L_j \cap L_k = \emptyset$ or if $r \geq 4$
    $L_i \cap L_j \cap L_k \neq \emptyset$ for some distinct $i, j, k$. 
\end{defi}
\begin{rem}
    It is natural to classify $3$-set with $L_i \cap L_j \cap L_k = \emptyset$ into non-intersecting type, but we regard such $3$-set as the trivial intersecting type for later convenience.
\end{rem}

\noindent Let us denote
\begin{equation}
    L_i = \bigcup_{j \in [r] \setminus \{ i \}} A_{i,j},
\end{equation}
where $A_{i,j} = L_i \cap L_j \subset [n]$. For $I \subset [r]$ with $2 \leq \lvert I \rvert \leq r-1$ we denote $A_I = \bigcap_{i \in I} L_i$ and $\lvert A_I \rvert = a_I$. We assume $a_{[r]} = 0$, and otherwise, by considering a restriction arrangement we obtain an arrangement with $r$-set $\TT' = \{ L_i' \}_{i=1, \dots, r}$ such that $\bigcap_{i=1}^r L_i' = \emptyset$, i.e., $a_{[r']} = 0$, $r' < r$. Since for $I, J \subset [r]$ we should have $\lvert A_I \rvert > \lvert A_J \rvert$ if $\lvert I \rvert < \lvert J \rvert$, we also assume $a_I > a_J$ if $\lvert I \rvert < \lvert J \rvert$. \\
Let us denote $\II_i^l = \{ I \subset [r] \mid \ \lvert I \rvert = l, i \in I \}$. \\
Since
\begin{equation*}
    L_i = \bigcup_{j \in [r] \setminus \{ i \}} A_{i,j},
\end{equation*}
we have
\begin{equation}\label{eq:Li}
    k+1 = \sum_{l=2}^{r-1} (-1)^l \sum_{I \in \II_i^l} a_I = \sum_{j \in [r] \setminus \{ i \}} a_{i,j} + \sum_{l=3}^{r-1} (-1)^l \sum_{I \in \II_i^l} a_I
\end{equation}
by set theoretic computation.
\begin{rem}
    If $a_I = 0$ for any $I$ with $\lvert I \rvert \geq 3$, we obtain the non-intersecting type $r$-sets explained in subsection \ref{sec:Benoit_type}.
\end{rem}
\noindent Since $A_{i,j} = L_j \setminus \bigcup_{l \in [r] \setminus \{ i,j \}} A_{l,j}$, we have
\begin{equation}\label{eq:a_ij_2}
    a_{1,j} = k+1 - \sum_{l \in [r] \setminus \{ 1,j \}} a_{l,j} - \sum_{l=3}^{r-1} (-1)^l \sum_{I \in \II_1^l} a_I.
\end{equation}
\noindent By summing both sides of the formula (\ref{eq:a_ij_2}) for all $j \in [r] \setminus \{ 1 \}$, we obtain
\begin{equation}
    \begin{split}
        &\sum_{j \in [r] \setminus \{ 1 \}} a_{1,j} = \sum_{j \in [r] \setminus \{ 1 \}} \left( k+1 - \sum_{l \in [r] \setminus \{ 1,j \}} a_{l,j} - \sum_{l=3}^{r-1} (-1)^l \sum_{I \in \II_j^l} a_I \right) \\
        & = (r-1)(k+1) - \sum_{j \in [r] \setminus \{ 1 \}} \sum_{l \in [r] \setminus \{ 1,j \}} a_{l,j} - \sum_{j \in [r] \setminus \{ 1 \}} \sum_{l=3}^{r-1} (-1)^l \sum_{I \in \II_j^l} a_I.
    \end{split}
\end{equation}
Thus, we  have the following relation.
\begin{equation}
    \begin{split}
        &k+1 - \sum_{l=3}^{r-1} (-1)^l \sum_{I \in \II_1^l} a_I \\
        & = (r-1) (k+1) - \sum_{j \in [r] \setminus \{ 1 \}} \sum_{l \in [r] \setminus \{ 1,j \}} a_{l,j} - \sum_{j \in [r] \setminus \{ 1 \}} \sum_{l=3}^{r-1} (-1)^l \sum_{I \in \II_j^l} a_I;
    \end{split}
\end{equation}
that is
\begin{equation}\label{eq:last_eq}
    \begin{split}
        & (r-2)(k+1) = \\
        & \sum_{j \in [r] \setminus \{ 1 \}} \sum_{l \in [r] \setminus \{ 1,j \}} a_{l,j} + \sum_{j \in [r] \setminus \{ 1 \}} \sum_{l=3}^{r-1} (-1)^l \sum_{I \in \II_j^l} a_I - \sum_{l=3}^{r-1} (-1)^l \sum_{I \in \II_i^l} a_I \\
        & = 2 \sum_{l,t \in [r] \setminus \{ 1 \}, l < t} a_{l,t} + \sum_{l=3}^{r-1} (-1)^l \left( \sum_{j \in [r] \setminus \{ 1 \}} \sum_{I \in \II_j^l} a_I - \sum_{I \in \II_1^l} a_I \right).
    \end{split}
\end{equation}

\noindent In terms of $r,k$, and $a_I$, the number of hyperplanes can be written as 
\begin{equation}
    \begin{split}
        \lvert \bigcup_{i=1}^r L_i \rvert & = \sum_{l=2}^{r-1} (-1)^l \sum_{j \in [r]} \sum_{I \in \II_j^l} a_I                                                                                      \\
        & = \sum_{[r] \setminus \{ 1 \}} a_{1,j} + \sum_{l,t \in [r] \setminus \{ 1 \}} a_{l,t} + \sum_{j \in [r]} \sum_{l=3}^{r-1}(-1)^l \sum_{I \in \II_j^l} a_I \\
        & =(r-1)(k+1) - \sum_{l,t \in [r] \setminus \{ 1 \}, l<t} a_{l,j} + \sum_{l=3}^{r-1} (-1)^l \sum_{I \in \II_1^l} a_I.
    \end{split}
\end{equation}

Unlike non-intersecting case, computing the number of hyperplanes would be complicated because of the terms $\sum a_{l,j}$ and $\sum_l (-1)^l \sum a_I$. It would be also difficult to give an explicit lower bound for $k$ which is the dimension of the space. These facts give us the difficulty to give an explicit classification of general $r$-sets of intersecting type as in Proposition \ref{prop:miniprop}.

Remark that once we fix numbers $\{ a_{l,t} \}_{l,t \in [r] \setminus \{ 1 \}, l<t}$, $\{ a_I \}_{I \subset [r], \mid I \mid \geq 3}$ satisfying (\ref{eq:last_eq}), the remaining numbers $\{ a_{1,j} \}_{j \in [r] \setminus \{ 1 \}}$ are uniquely determined by equation (\ref{eq:a_ij_2}), and thus, the $r$-set $\TT = \{ L_1, \dots, L_r \}$ is uniquely determined up to renumbering of elements in $\bigcup_{i=1}^r L_i$ or indices of the sets $L_i$. 

Though it is difficult task to give an explicit classification of general intersecting $r$-sets as described in Section 4.1, it includes an important class of $r$-sets which already appeared in \cite{LS}, \cite{SSY1}, \cite{SSY2}.
Let us see such examples.

\begin{ex}[Good $3s$-partition]\label{ex:good3s}
    Let us consider $r=3$. In this case we have $\lvert \bigcup_{i=1}^3 L_i \rvert = \frac{3(k+1)}{2}$. Since the number $\frac{3(k+1)}{2}$ corresponds to the number of hyperplanes, $k$ should be a positive odd number, so we can write it as $k = 2s-1$. Since in the case of $s = 1$ non-very generic arrangement does not appear (in this case we have braid arrangement with three hyperplanes $\{ H_{i,j} \}_{1 \leq i<j \leq 3}$, which is very generic), we have to assume $s \geq 2$.  \\
    We have $\lvert L_i \rvert = k + 1 = 2s$ and $\lvert L_i \cap L_j \rvert = s$ since there is only one correspondence
    \begin{equation}
        \frac{(3-2)(k+1)}{2} = s = a_{2,3} \leftrightarrow (a_{1,2}, a_{1,3}, a_{2,3}) = (s,s,s) .
    \end{equation}
    Remark that the 3-set $\TT = \{ L_1, L_2, L_3 \}$ given in this example is the good $3s$-partition first considered in \cite{LS} and developed in \cite{SSY1}.
\end{ex}

\begin{ex}[4-set of intersecting type]\label{ex:good12}
    Equation (\ref{eq:last_eq}) would be
    \begin{equation}
        18 = 2(a_{2,3} + a_{2,4} + a_{3,4}) - (a_{1,2,3} + a_{1,2,4} + a_{1,3,4}) - 3a_{2,3,4}.
    \end{equation}
    Consider the following tuple
    \begin{equation}
        (a_{2,3},a_{2,4},a_{3,4},a_{1,2,3},a_{1,2,4},a_{1,3,4},a_{2,3,4}) = (6,6,6,3,3,3,3).
    \end{equation}
    Then the other corresponding numbers are
    \begin{equation}\label{eq:MS(12,8)}
        (a_{1,2},a_{1,3},a_{1,4}) = (3,3,3).
    \end{equation}
    Then we have a 4-set $\TT = \{ L_1, L_2, L_3, L_4 \}$ with $L_1 = [12] \setminus \{ 10,11,12 \}, L_2 = [12] \setminus \{ 7,8,9 \}, L_3 = [12] \setminus \{ 4,5,6 \}, L_4 = [12] \setminus \{ 1,2,3 \}$ for example. Notice that this 4-set is assumed in Example \ref{ex:MS(12,8)} (see also \cite{SY}).
\end{ex}

\noindent The following proposition is a generalization of Examples \ref{ex:good3s} and \ref{ex:good12}.

\begin{prop}\label{ex:goodrs}
    Let us fix $s \geq r - 1$. For subsets $K_i \subset [rs]$ such that $\lvert K_i \rvert = \lvert K_j \rvert = s$, $K_i \cap K_j = \emptyset$ for any $i,j$ and $\bigcup_{i=1}^r K_i = [rs]$ define $L_i = [rs] \setminus K_i$. Then, the equation (\ref{eq:last_eq}) holds.
\end{prop}
\begin{proof}
    The proof is by direct computation. Since
    \begin{align*}
        (r - 2)(k + 1) = (r - 2)(r - 1)s
    \end{align*}
    by (\ref{eq:a_ij_2}), while
    \begin{align*}
         & (r-2)(k+1) = 2 \sum_{l,t \in [r] \setminus \{ 1 \}, l < t} a_{l,t} + \sum_{l=3}^{r-1} (-1)^l \left( \sum_{j \in [r] \setminus \{ 1 \}} \sum_{I \in \II_j^l} a_I - \sum_{I \in \II_1^l} a_I \right) \\
         & = 2\binom{r - 1}{2} (r-2)s + \sum_{l=3}^{r-1} (-1)^l \binom{r-1}{l-1}(r-l)s \left( r - 2 \right) = (r-2)(r-1)s.
    \end{align*}
    Thus, the equation (\ref{eq:last_eq}) holds.
\end{proof}

\begin{rem}
    For non-intersecting type $\TT$, we gave an explicit classification as in Proposition \ref{prop:miniprop}, while for intersecting type $\TT$ it would be complicated task to give such a classification. The $r$-set assmued in Proposition \ref{ex:goodrs} would be the first classification for intersecting type $r$-set, which is one of the main topics in Section \ref{sec:good_rs}.
\end{rem}
%%%%%%%%%%%%%%%%%%%%%%%%%%%%%%%%%%%%%%%%%%%%%%%%%%%%%%%%%%%%%%%%%%%%%%
\section{A sufficient conditions for non-very genericity}\label{sec:good_rs}
For a fixed $d \in \ZZ$ consider an $r$-set $\TT = \{ L_1, \dots, L_r\}$ and sets of vectors $\{ v_{i,j}^t \}_{i=2, \dots, r}$, $t=1, \dots, d$ where each vector satisfies $v_{i,j}^t= v_{1,j}^t - v_{1,i}^t$. 

In this section let us give sufficient conditions for $\A^0$ to be non-very generic by giving conditions on the certain numbers of the sets $\{ v_{1,i}^t \}_{i=2, \dots, r}$, $t=1, \dots, d$ and showing they are sufficient conditions to be $K_\TT$-vector set.
The conditions would be tractable and easy ones to check by hand. 

\begin{nota}
    We denote $a_I = \lvert \bigcap_{i \in I} L_i \rvert$. If $\lvert I \rvert = 2$, we also denote $a_{i,j} = \lvert L_i \cap L_j \rvert$ for specifically.
\end{nota}

\paragraph{The case of the $r$-set $\TT$ being non-intersecting type}
\noindent For $1 \leq a<b \leq r$ let us denote by
\begin{align*}
    V_{a,b} = \left< v_{a,b}^t \mid t = 1, \dots, d \right>
\end{align*}
the vector space spanned by vectors $v_{a,b}^t \in H_{a,b}^0$, $t = 1, \dots, d$. 
Since vectors $v_{a,b}^t$, $t = 1, \dots, d$ are contained in $\displaystyle H_{a,b}^0$ for any $a,b$, it follows that
\begin{equation}
    H_{a,b}^0 \supset V_{a,b},
\end{equation}
equivalently,
\begin{equation}
    \left< \alpha_p \mid p \in L_a \cap L_b  \right> = \sum_{p \in L_a \cap L_b} \left( H_p^0 \right)^\perp \subset V_{a,b}^\perp.
\end{equation}
The following proposition holds.

\begin{prop}\label{prop:main}
    Let us fix $n = \frac{r(k+1)}{2}$ and $\A^0 = \{ H_i^0 \}_{i = 1, \dots, n}$ be an arrangement in $\CC^k$ with normal vectors $\alpha_i$. Let $\TT = \{ L_1, \dots, L_r \}$ be an $r$-set of non-intersecting type. If there exist $d \geq n-k-r+1$ sets of vectors $\{ v_{1,i}^t \}_{i=2, \dots, r}$, $t = 1, \dots, d$ satisfying
    \begin{equation}\label{eq:non-int}
        \dim \sum_{(i,j) \in D([r])} V_{i,j}^\perp
        \begin{cases}
            \geq \displaystyle \sum_{(i,j) \in D([r])} a_{i,j}  \mbox{ if $\displaystyle \sum_{(i,j) \in D([r])} a_{i,j} < k$}, \quad \text{and}
            \\
            = k                                   \mbox{ if $\displaystyle \sum_{(i,j) \in D([r])} a_{i,j} \geq k$},
        \end{cases}
    \end{equation}
    then $\A^0$ is non-very generic.
\end{prop}
\begin{proof}
    Let $\{ v_{1,i}^t \}_{i=2, \dots, r}$, $t = 1, \dots, d$ satisfy (\ref{eq:non-int}). Then, there exist $\sum_{(i,j) \in D([r])} a_{i,j}$ independent vectors in $\sum_{(i,j) \in D([r])} V_{i,j}^\perp$ if $\sum_{(i,j) \in D([r])} a_{i,j} < k$, and $k$ independent vectors if $\sum_{(i,j) \in D([r])} a_{i,j} \geq k$. Thus, we can choose generic vectors $\alpha_l$, $l = 1, \dots, n$ from orthogonal spaces $V_{a,b}^\perp$, $1 \leq a < b \leq r$. That is we obtain a generic arrangement $\A^0$. 
    
    For each translation $\A^t$ of $\A^0$ consider $r-1$ vectors $v_{i,j}^t$, $j \neq i$. Since $v_{i,j}^t \in H_{i,j}^0$, we have $\bigcap_{k \neq i} H_{i,k}^t = \bigcap_{p \in L_i} H_p^t \neq \emptyset$ which gives an intersection point $P_i^t$ of exactly $k+1$ hyperplanes. Thus, the translation $\A^t$ is $K_\TT$-translated, i.e., the sets $\{ v_{1,i}^t \}_{i=2, \dots, r}$, $t = 1, \dots, d$ are $K_\TT$-vector sets. \\
    If the $K_\TT$-vector sets are linearly dependent, by replacing some vector $v_{1,i}^t$ with its multiple we can assume that the $K_\TT$-vector sets are independent a priori. Notice that the assumption $d \geq n - k - r + 1$ satisfies the one in Theorem \ref{thm:main2}. \\
    Therefore, Theorem \ref{thm:main2} implies that if there exist $d \geq n - k - r + 1$ sets $\{ v_{1,i}^t \}_{i=2, \dots, r}$, $t = 1, \dots, d$ of vectors $v_{1,i}^t$ satisfying (\ref{eq:non-int}), then $\A^0$ is non-very generic.
\end{proof}

\begin{rem}\label{rem:main-1}
    Let $\A^0$ be an arrangement of $n \left( > \frac{r(k+1)}{2} \right)$ hyperplanes in $\CC^k$, and $\TT = \{ L_1, \dots, L_r \}$ be an $r$-set of non-intersecting type satisfying $\bigcup_{i=1}^r L_i \subset [n]$. Let $\B^0 \subset \A^0$ be a subarrangement consists of hyperplanes indexed in $\bigcup_{i=1}^r L_i$. If $\B^0$ is non-very generic, then $\A^0$ is non-very generic. Analogously, if there exists a restriction arrangement $\left( \A^0 \right)^{Y_{\B^0}} = \{ H^0 \cap Y_{\B^0} \mid  H^0 \in \A^0 \setminus \B^0 \}$, $Y_{\B^0} = \bigcap_{H^0 \in \B^0} H$ of $\A^0$ which is non-very generic, then $\A^0$ is non-very generic.
    
\end{rem}
By Proposition \ref{prop:main} and Remark \ref{rem:main-1} we have the following theorem in which we do not assume $n = \frac{r(k+1)}{2}$.
\begin{thm}\label{thm:main}
    Let $\A^0 = \{ H_i^0 \}_{i=1, \dots, n}$ be an arrangement in $\CC^k$ and $\TT = \{ L_1, \dots, L_r \}$ be an $r$-set $(r \geq 4)$ of non-intersecting type such that $\lvert \bigcup_{i=1}^r L_i \rvert = \frac{r(k+1)}{2}$. If there exists a restriction $\left( \A^0 \right)^{Y_{\B^0}}$, $Y_{\B^0} = \bigcap_{H \in \B^0} H$, $\B^0 = \left\{ H_p \mid p \in \bigcup_{i=1}^r L_i \right\}$ with $d \geq \frac{r(k+1)}{2} -k-r+1$ the sets $\{ v_{1,i}^t \}_{i=2, \dots, r}$, $t = 1, \dots, d$ satisfying (\ref{eq:non-int}), then $\A^0$ is non-very generic.
\end{thm}
According to Theorem \ref{thm:main}, if $\TT$ is the non-intersecting $r$-set, then $\frac{r(k+1)}{2}$ is the minimum number of hyperplanes which give non-very generic arrangement. Thus, as a corollary of Theorem \ref{thm:main} we have following.
\begin{cor}\label{cor:lib}
    Let $\TT$ be an $r$-set of non-intersecting type and $n < \frac{r(k+1)}{2}$. Any arrangement of $n$ hyperplanes is very generic.
\end{cor}

\paragraph{The case of the $r$-set $\TT$ being intersecting type}
An example of a non-very generic arrangement with a $3$-set of intersecting type first appeared in \cite{Falk} and then was further developed in \cite{LS}. Based on work in \cite{LS}, the authors defined a good $3s$-partition in \cite{SSY1} and showed that it gives rise to non-very generic arrangements. Let us introduce the \textit{good $rs$-partition} as a generalization of the good $3s$-partition and show that it gives rise to non-very generic arrangements. This would be the first classification of intersecting type $r$-sets.

\begin{defi}
    Let us fix $s \geq r - 1$. For subsets $K_i \subset [rs]$ such that $\lvert K_i \rvert = \lvert K_j \rvert = s$, $K_i \cap K_j = \emptyset$ for any $i,j$ and $\bigcup_{i=1}^r K_i = [rs]$, define $L_i = [rs] \setminus K_i$. We call the set $\TT= \{ L_1, \dots, L_r \}$ a \textit{good $rs$-partition}.
\end{defi}

Let $\A^0 = \{ H_i^0 \}_{i = 1, \dots, n}$ be an arrangement in $\CC^k$ with normal vectors $\alpha_i$, $i=1, \dots, n$ and $\TT = \{ L_1, \dots, L_r \}$ be a good $rs$-partition with $L_i = [rs] \setminus K_i$, $K_i = \{ (r - i)s + 1, \dots, (r - i + 1)s \}$, $i = 1, \dots, r$. We assume $n = rs$ for a while in this section. \\
For $l=1,\dots,r$ denote by
\begin{align*}
    V_{[r] \setminus \{ l \}} = \left< v_{a,b}^t \mid a,b \in [r] \setminus \{ l \}, t = 1, \dots, d \right>
\end{align*}
the vector space spanned by vectors $v_{a,b}^t$, $a,b \in [r] \setminus \{ l \}$, $t = 1, \dots, d$. Since for fixed $l$ the vectors $v_{a,b}^t$ are contained in $\displaystyle H_{a,b}^0$, it follows that
\begin{equation}
    \bigcap_{p \in \bigcap_{t \in [r] \setminus \{ l \}} L_t} H_p^0 \supset V_{[r] \setminus \{ l \}}
\end{equation}
and equivalently,
\begin{equation}
    \left< \alpha_p \mid  p \in \bigcap_{t \in [r] \setminus \{ l \}} L_t  \right> = \sum_{p \in \bigcap_{t \in [r] \setminus \{ l \}} L_t} \left( H_p^0 \right)^\perp \subset V_{[r] \setminus \left\{ l \right\}}^\perp.
\end{equation}
The following proposition holds.
\begin{prop}\label{lem:mainprop}
    Let $\A^0 = \{ H_i^0 \}_{i=1, \dots, rs}$ be an arrangement in $\CC^{(r-1)s-1}$ with normal vectors $\alpha_i$ and let $\TT = \{ L_1, \dots, L_r \}$ be a good $rs$-partition. If there exist $d \geq s-r+2$ sets of vectors $\{ v_{1,i}^t \}_{i=2, \dots, r}$, $t = 1, \dots, d$ satisfying
    \begin{equation}\label{eq:int}
        \dim \sum_{l \in I} V_{[4] \setminus \{ l \}}^\perp
        \begin{cases}
            \geq s \lvert I \rvert \quad  \text{for any $I \subset [r], 1 \leq \lvert I \rvert \leq r-2$},  \text{and} \\
            = (r-1)s-1 \quad  \text{for any $I \subset [4], \lvert I \rvert = r-1$},
        \end{cases}
    \end{equation}
    then $\A^0$ is non-very generic.
\end{prop}
\begin{proof}
    Let $\{ v_{1,i}^t \}_{i = 2, \dots, r}$, $t = 1, \dots, d$ satisfy the condition (\ref{eq:int}). Then, by the similar observation in proof of Proposition \ref{prop:main}, we can choose generic vectors $\alpha_l$, $l = 1, \dots, n$ from the orthogonal spaces $V_{[r] \setminus \{ l \}}^\perp$, $l = 1, \dots, r$. That is we obtain a generic arrangement $\A^0$. Moreover, as similar as the observation in proof of Proposition \ref{prop:main} we can also see the translations $\A^t$, $t = 1, \dots, d$ are $K_\TT$-translations, i.e., the sets $\{ v_{1,i}^t \}_{i = 2, \dots, r}$, $t = 1, \dots, d$ are linearly independent $K_\TT$-vector sets. Notice that the condition $d \geq s-r+2$ satisfies the assumption of Theorem \ref{thm:main2} since $d \geq n-k-r+1 = rs - \{ (r-1)s - 1 \} -r + 1 = s-r+2$.
    
    Therefore, Theorem \ref{thm:main2} implies that if there exist $d \geq s -r + 2$ sets of vectors $\{ v_{1,i}^t \}_{i=2, \dots, r}$, $t = 1, \dots, d$ satisfying the condition (\ref{eq:int}), then $\A^0$ is non-very generic.
\end{proof}

\begin{rem}\label{rem:main}
    Let $\A^0$ be an arrangement of $n (> rs)$ hyperplanes in $\CC^k$, and $\TT = \{ L_1, \dots, L_r \}$ be a good $rs$-partition satisfying $\bigcup_{i=1}^r L_i \subset [n]$. Let $\B^0 \subset \A^0$ be a subarrangement consists of hyperplanes indexed in $\bigcup_{i=1}^r L_i$. If $\B^0$ is non-very generic, then $\A^0$ is non-very generic. Analogously, if there exists a restriction arrangement $\left( \A^0 \right)^{Y_{\B^0}} = \{ H^0 \cap Y_{\B^0} \mid  H^0 \in \A^0 \setminus \B^0 \}$, $Y_{\B^0} = \bigcap_{H^0 \in \B^0} H$ of $\A^0$ which is non-very generic, then $\A^0$ is non-very generic.
\end{rem}
By Proposition \ref{lem:mainprop} and Remark \ref{rem:main} we have the following theorem in which we do not assume $n = rs$.

\begin{thm}
    Let $\A^0 = \{ H_i^0 \}_{i=1, \dots, n}$ be an arrangement in $\CC^k$ and $\TT = \{ L_1, \dots, L_r \}$ be a good $rs$-partition. If there exists a restriction $\left( \A^0 \right)^{Y_{\B^0}}$, $Y_{\B^0} = \bigcap_{H \in \B^0} H$, $\B^0 = \left\{ H_p^0 \mid p \in \bigcup_{i=1}^r L_i \right\}$ with $d \geq s-r+2$ the sets $\{ v_{1,i}^t \}_{i=2, \dots, r}$, $t = 1, \dots, d$ satisfying (\ref{eq:int}), then $\A^0$ is non-very generic.
\end{thm}
As a generalization of Example \ref{ex:MS(12,8)}, let us see how we find the sets $\{ v_{1,i}^t \}_{i=2, \dots, r}$, $t = 1, \dots, d$. By Remark \ref{rem:main} it suffices to consider arrangement $\A^0$ of $rs$ hyperplanes in $\CC^{(r-1)s-1}$. Let $d_{1,i} = \dim \left< v_{1,i}^t \mid t = 1, \dots, d \right>$, $i = 2, \dots, r$. \\
The following theorem gives an explicit way how to find $K_\TT$-vector sets when the $r$-set is the good $rs$-partition.
\begin{thm}\label{thm:mainlast}
    the sets $\{ v_{1,i}^t \}_{i=2, \dots, r}$, $t = 1, \dots, d$ satisfy (\ref{eq:int}) if and only if $v_{1,a}^l \in \left<v_{1,a}^t \mid t = 1, \dots, d \setminus \{ l \} \right>$ for any $a \in \{ 2,\dots, r \}$, $l \in [d]$ and $\sum_{i \in I} d_{1,i} \leq (r- \lvert I \rvert - 1)s - 1$ for any $I \subset [r], 1 \leq \lvert I \rvert \leq r-2$.
\end{thm}
\begin{proof}
    First, let us prove that if the sets $\{ v_{1,i}^t \}_{i=2, \dots, r}$, $t = 1, \dots, d$ satisfy (\ref{eq:int}), then $v_{1,a}^l \in \left<v_{1,a}^t \mid t = 1, \dots, d \setminus \{ l \} \right>$ for any $a \in \{ 2,\dots, r \}$, $l \in [d]$. We prove this in the case of $a=2$ by contradiction assuming that there exists a vector $v_{1,2}^j$ such that $v_{1,2}^j \in \left< v_{1,i}^t \mid i=3, \dots, r, t = 1, \dots, d \right>$. In this case we have $v_{1,2}^j \in V_{[r] \setminus \{ 2 \}}$. \\
    By (\ref{eq:int}) we have
    \begin{equation}\label{eq:hara}
        \dim \sum_{l = 2}^r V_{[r] \setminus \{ l \}}^\perp = (r-1)s-1 \iff \dim \bigcap_{l=2}^r V_{[r] \setminus \{ l \}} = 0.
    \end{equation}
    On the other hand, since we have $v_{1,2}^j \in V_{[r] \setminus \{ l \}}$, $l=3, \dots, r$ and also $v_{1,2}^j \in V_{[4] \setminus \{ 2 \}}$ by assumption, it follows that $v_{1,2}^j \in \bigcap_{l=2}^r V_{[r] \setminus \{ l \}}$, which contradicts the fact $dim \bigcap_{l=2}^r V_{[r] \setminus \{ l \}} = 0$. Thus, we have $v_{1,2}^j \in \left< v_{1,2}^t \mid t = 1, \dots, d \setminus \{ j \} \right>$. The analogous proofs also follow for any $a \in \{ 2,\dots,r \}$ and $l \in [d]$. \\
    Secondly, let us prove $\sum_{i \in I} d_{1,i} \leq (r- \lvert I \rvert - 1)s - 1$ for any $I \subset [r], 1 \leq \lvert I \rvert \leq r-2$. By the fact we have just proved we obtain
    \begin{equation}\label{eq:kimochi}
        \left< v_{1,i}^t \mid i = 2, \dots, m, \ t = 1, \dots, d \right> \cap \left< v_{1,m+1}^t \mid t = 1, \dots, d \right> = \{ 0 \}
    \end{equation}
    for any $2 \leq m \leq r-1$. \\
    On the other hand, we have
    \begin{equation}
        \displaystyle \dim \sum_{l \in I} V_{[r] \setminus \{ l \}}^\perp \geq s \lvert I \rvert \iff \dim \bigcap_{l \in I} V_{[r] \setminus \{ l \}}
        \leq (r- \lvert I \rvert - 1)s - 1
    \end{equation}
    for any $I \subset [r], 1 \leq \lvert I \rvert \leq r-2$ by (\ref{eq:int}); thus we have $\sum_{i \in I} d_{1,i} \leq (r- \lvert I \rvert - 1)s - 1$ for any $I \subset [r], 1 \leq \lvert I \rvert \leq r-2$. \\
    Conversely, let assume $\sum_{i \in I} d_{1,i} \leq (r- \lvert I \rvert - 1)s - 1$ for any $I \subset [r], 1 \leq \lvert I \rvert \leq r-2$. By considering (\ref{eq:hara}), if there exist the sets satisfying the assumptions, then the proof would be completed. For this reason it is sufficient to show there exist such the sets. \\
    Since $v_{1,a}^l \in \left<v_{1,a}^t \mid t = 1, \dots, d \setminus \{ l \} \right>$ for any $a \in \{ 2,\dots, r \}$, $l \in [d]$, we have (\ref{eq:kimochi}). In particular, we have
    \begin{align*}
        \dim \left< v_{1,i}^t \mid i=2, \dots, r, t = 1, \dots, d \right> = \sum_{i=2}^r \dim \left< v_{1,i}^t \mid t = 1, \dots, d \right>  = \sum_{i=2}^r d_{1,i}.
    \end{align*}
    To construct the the sets $\{ v_{1,i}^t \}_{i=2, \dots, r}$, $t = 1, \dots, d$ we need to choose $(r-1)d$ vectors $v_{1,i}^t, i = 2,\dots,r, t = 1, \dots, d$ with $\dim \left< v_{1,i}^t \mid t = 1, \dots, d \right> = d_{1,i}$, $i = 2,\dots,r$. In particular, it is sufficient to choose $\sum_{i=2}^r d_{1,i}$ independent vectors in $\CC^{(r-1)s-1}$. \\
    By assumption we have $\sum_{i \in I} d_{1,i} \leq (r- \lvert I \rvert - 1)s - 1$ for any $I \subset [r], 1 \leq \lvert I \rvert \leq r-2$. \\
    Since the inequality
    \begin{equation*}
        d_{1,i} \leq s-1
    \end{equation*}
    holds for any $i$ when $\lvert I \rvert = r-2$, and
    \begin{equation*}
        (k-1)d_{1,i} \leq (k-1)(s-1) \leq (k-1)s - 1
    \end{equation*}
    holds for any $1 \leq r-k \leq r-2$, we have $\sum_{i=2}^{r} d_{1,i} \leq (r-1)(s-1) < (r-1)s - 1$, and thus the the sets we expected actually exist.
\end{proof}

According to Theorem \ref{thm:mainlast}, if $\TT$ is the good $rs$-partition, then $rs$ is the minimum number of hyperplanes which give non-very generic arrangement. Thus, as a corollary of Theorem \ref{thm:mainlast} we have following.
\begin{cor}\label{cor:lib2}
    Let $\TT$ be a good $rs$-partition and $n < rs$. Any arrangement of $n$ hyperplanes is very generic.
\end{cor}

\begin{rem}
    When $\TT$ is a general case of intersecting type, determining a sufficient condition for the sets $\{ v_{1,i}^t \}_{i = 2, \dots, r}$, $t = 1, \dots, d$ to be $K_\TT$-vector sets seems to be dependent on the manner in which tuples $(a_{i,j})_{i,j}$ are fixed. By this reason giving explicit conditions for the sets $\{ v_{1,i}^t \}_{i \neq 1}$, $t = 1, \dots, d$ to be $K_\TT$-vector sets is remain open problem, when $\TT$ is not good $rs$-partition but intersecting type.
\end{rem}
%%%%%%%%%%%%%%%%%%%%%%%%%%%%%%%%%%%%%%%%%%%%%%%%%%%%%%%%%%%%%%%%%%%%%%
\noindent \textit{Acknowledgements.}
The author would like to thank Anatoly Libgober for useful discussions and for pointing out Corollary \ref{cor:lib} and Corollary \ref{cor:lib2}, and Masahiko Yoshinaga for useful discussions. The author was supported by JSPS Research Fellowship for Young Scientists Grant Number 20J10012.
%%%%

%%%%%%%%%%%%%%%%%%%%%%%%%%%%%%%%%%%%%%%%%%%%%%%%%%%%%%%%%%%%%%%%%%%%%%

\end{document}